\documentclass[11pt]{article}
\usepackage{graphicx}
\usepackage{amsfonts}
\usepackage{amsmath}
\usepackage{amssymb}

\usepackage{url}
\usepackage{wrapfig}
\usepackage{float}
\usepackage{fancyhdr}
\usepackage{indentfirst}
\usepackage{enumerate}
\usepackage{epstopdf}
\usepackage{dsfont}
 \usepackage[colorlinks=true,citecolor=blue]{hyperref}
\usepackage{amsthm}
\usepackage{multirow}
\usepackage{color}
\usepackage{natbib}
\usepackage{comment}
\usepackage{fancyhdr}
\usepackage{indentfirst}
\usepackage{dsfont,footmisc}
\usepackage{setspace}  

\usepackage{tikz}

\usepackage{caption}

\usetikzlibrary{decorations.pathreplacing}

\usetikzlibrary{cd} 
\usepackage{tkz-graph}
\usepackage{pgfplots}
\usepackage{subcaption}
\usetikzlibrary{calc}

\usepgfplotslibrary{fillbetween}

\usetikzlibrary{patterns, positioning, arrows}
\usetikzlibrary{arrows.meta}

\usepackage{schemabloc}
\usetikzlibrary{decorations.markings}

\def\d{\mathrm{d}}
\def\laweq{\buildrel \mathrm{d} \over =}

\newcommand{\X}{\mathcal {X}}
\newcommand{\C}{\mathcal {C}}

\newcommand{\T}{\mathcal T}
\newcommand{\VaR}{\mathrm{VaR}}
\newcommand{\RVaR}{\mathrm{RVaR}}
\newcommand{\ES}{\mathrm{ES}}

\newcommand{\E}{\mathbb{E}}
\newcommand{\EE}{\mathcal{E}}
\newcommand{\B}{\mathcal{B}}

\newcommand{\R}{\mathbb{R}}

\newcommand{\p}{\mathbb{P}}

\newcommand{\U}{\mathrm{U}}

\newcommand{\id}{\mathds{1}}

\renewcommand{\(}{\left(}
\renewcommand{\)}{\right)}

\renewcommand{\ge}{\geqslant}
\renewcommand{\le}{\leqslant}

\renewcommand{\leq}{\leqslant}
\renewcommand{\epsilon}{\varepsilon}

\def\lawis{\buildrel \mathrm{d} \over \sim}

\theoremstyle{plain}
\newtheorem{theorem}{Theorem}

\newtheorem{lemma}{Lemma}
\newtheorem{proposition}{Proposition}
\theoremstyle{definition}
\newtheorem{definition}{Definition}
\newtheorem{example}{Example}
\usepackage{tkz-graph}
\usetikzlibrary{patterns, positioning, arrows}
\usetikzlibrary{decorations.markings}

\theoremstyle{remark}
\newtheorem{remark}{Remark}

\newcommand{\cet}{\begin{center}}
\newcommand{\ecet}{\end{center}}

\pgfplotsset{compat=1.17}

\begin{document}

\title{Distorted optimal transport}

\author{Haiyan Liu\thanks{Department of Mathematics and Department of Statistics and Probability, Michigan State University, United States. E-mail: \url{liuhaiy1@msu.edu}.} \and  Bin Wang\thanks{Academy of Mathematics and Systems Science, Chinese Academy of Sciences, Beijing 100190, China. E-mail: \url{wangbin@amss.ac.cn}.} \and Ruodu Wang\thanks{Department of Statistics and Actuarial Science, University of Waterloo, Waterloo, ON N2L3G1, Canada. E-mail: \url{wang@uwaterloo.ca}.} \and Sheng Chao Zhuang\thanks{Department of Finance, University of Nebraska-Lincoln, Lincoln, Nebraska, 68588, United States. E-mail: \url{szhuang3@unl.edu}.}}

\maketitle

\begin{abstract}
Classic optimal transport theory is formulated through  minimizing  the expected transport cost   between two given distributions. 
We propose the framework of distorted optimal transport by minimizing a distorted expected cost, which is the cost under a non-linear  expectation.
This new formulation is motivated by concrete problems in decision theory, robust optimization, and risk management,
and it has many distinct features compared to the classic theory.
We choose simple cost functions and study  different distortion functions and their implications on the optimal transport plan. 
We show that on the real line, the comonotonic coupling is optimal for the distorted optimal transport problem when  the distortion function is convex and the cost function is submodular  and monotone.
Some  forms of duality and uniqueness results are provided.  
For inverse-S-shaped distortion functions and linear cost, we obtain the unique form of optimal coupling for all marginal distributions, which turns out to have an interesting ``first comonotonic, then counter-monotonic" dependence structure; for S-shaped distortion functions a similar structure is obtained. Our results highlight several challenges and features in distorted optimal transport, offering a new mathematical bridge between the fields of probability,  decision theory, and risk management.

	\bigskip
	
	\end{abstract}

\newpage	

\section{Introduction}
 
  Optimal transport theory, originally developed by Monge and Kantorovich (see \cite{V09} for a history), has been an important field of study over the past several decades, by connecting
 various scientific fields, including probability, economic theory, operations research, statistics, machine learning, and quantitative finance.  For a mathematical background on optimal transport and its applications, we refer to the monographs of \cite{S15} and \cite{V09}.  
A specialized treatment of optimal transport in economics is given by \cite{G16}.

Classic optimal transport theory concerns the minimization of the expected cost  of transporting between two marginal distributions, that is, using a probabilistic formulation, 
 $$
\mbox{to minimize } \E[c(X,Y)] \mbox{ subject to $X\lawis \mu$ and $Y\lawis \nu$,}
 $$
where  $\mu$ and $\nu$ are two probability measures on suitable spaces, $X\lawis \mu$ means that $X$ has the distribution $\mu$ under a fixed probability measure $\p$ with respect to which the expectation $\E$ is evaluated, and $c$ is a real-valued cost function. The joint distribution of $(X,Y)$  is called a transport plan or a coupling in the literature. 

In this paper, we consider a novel setting of 
minimizing \emph{non-linear expectations} (e.g., \cite{P19}) of the transport cost.
 With a non-linear expectation $\EE$ replacing $\E$ in the classic formulation, the corresponding optimal transport problem is
 \begin{equation}\label{eq:R1-1}
\mbox{to minimize } \EE[c(X,Y)] \mbox{ subject to $X\lawis \mu$ and $Y\lawis \nu$.}
 \end{equation}
We will mainly focus on  a popular and tractable class of non-liner expectations,  the 
\emph{distorted expectations}, denoted by $\EE^h$, where $h$ is a \emph{distortion function} (see Section \ref{sec:motivation} for its definition). The \emph{distorted optimal transport} problem is \eqref{eq:R1-1} with $\EE=\EE^h$.
The non-linearity of $\EE^h$ leads to many mathematical challenges that are not treated in the classic theory.

 In addition to its mathematical novelty, 
 we explain in Section \ref{sec:motivation}
 three different motivations for using  non-linear and distorted expectations in optimal transport, which arise from the fields of decision theory (e.g., \cite{Q82, Y87}), robust optimization  (e.g., \cite{GS89, GS10}) and risk management (e.g., \cite{R13, EPRWB14}), respectively.
 In particular, non-linear and distorted expectations are useful tools to model preferences under risk, regulatory risk measures, and   transport cost under uncertainty. Thus, the study of distorted transport brings together ideas  and techniques from these fields.

As the first study on distorted optimal transport, we will focus on the basic setting of transport between two distributions on the real line.
We work mainly with simple forms of cost functions, and analyze different distortion functions and their implications on the optimal transport problem. The optimal couplings will be described by using copulas (\cite{N06}), which are convenient in our setting. 
 As a well-known result in classic optimal transport theory, 
if the cost function on the real plane is submodular, then an optimal transport plan is the  comonotonic coupling. 
Moreover, such a transport plan is \emph{universally optimal}, in the sense that its optimality does not depend on the marginal distributions. 
In the setting of distorted transport, 
we show  that 
if the distortion function is convex and the cost function is submodular and monotone (i.e., increasing component-wise  or decreasing component-wise), then a universally optimal transport plan is the comonotonic coupling (Theorem \ref{th:general}).  
This result highlights two distinct features of distorted optimal transport compared to the classic setting.
First,  convexity of the distortion function $h$ is essential for the optimal transport plan. The classic setting corresponds to a linear $h$, thus both convex and concave. 
Second, monotonicity of the cost function is important  for the optimal transport plan. 
Note that  monotonicity is irrelevant to the classic transport problem, since the optimal coupling is invariant if we add to the cost function some terms that depend only on one of the variables. This is not true in the distorted transport setting, because the distorted expectation is not linear, and additive terms that only depend on one variable generally affect the optimal coupling.  
Moreover, we establish a weak version of duality for convex $h$, but strong duality  holds  only in case of an affine cost  (Theorem \ref{th:dual}). 
These results are discussed in Section \ref{sec:3}. 
 
We proceed to concentrate on the linear cost $c(x,y)=x+y$, which is   submodular, supermodular, and monotone, in Sections \ref{sec:4}-\ref{sec:RA}.
The linear cost represents an additive structure, and it is the standard model in decision making with background risk (see e.g., \cite{Pratt1988}) and risk aggregation problems (e.g., \cite{EPR13}). 
 This cost function is trivial in the classic transport theory, but it yields some clean and nontrivial results in the framework of distorted optimal transport.
With the linear cost, we  show uniqueness of the universally optimal coupling based on strict convexity or concavity of $h$ (Theorem \ref{th:unique}). 
This and some other general results on the linear cost are presented in Section \ref{sec:4}. 
 In decision theory, abundant research has elicited  distortion functions from experimental data of human behavior. 
A special form of distortion functions that is  empirically plausible is the inverse S-shaped  distortion functions (\cite{E54}),  which plays a central role in cumulative prospect theory (\cite{TK92}).\footnote{Kahneman was awarded the Nobel Memorial Prize in Economic Sciences in 2002 ``for having integrated insights from psychological research into economic science, especially concerning human judgment and decision-making under uncertainty."}
We pay special attention to this class of distortion functions in Section \ref{sec:ISS}. 
 Such distortion functions are neither convex nor concave, leading to considerable technical challenges.
In this setting, we obtain equivalent conditions for the existence of {universally optimal} couplings, which turn out 
to have  an interesting ``first comonotonic then counter-monotonic" structure (Theorem \ref{th:2}).   
In contrast to inverse-S-shaped distortion functions in decision theory, 
 S-shaped distortion functions are closely related to   popular risk measures   in risk management. Universally optimal couplings for S-shaped distortion functions (Theorem \ref{pr:th2dual}) and its connection to robust risk aggregation are discussed in Section \ref{sec:RA}.
  
The framework of distorted optimal transport is new and many questions remain unanswered. 
Section \ref{sec:conclusion} concludes the paper with discussions on several open questions and future research directions. 

\section{Non-linear and distorted expectations}

 We first define non-linear and distorted expectations. Fix an atomless  probability space $(\Omega, \mathcal A, \p)$, and let $\X$ be the set of bounded random variables.   
A \emph{non-linear expectation} $\EE$ is a mapping from $\X$ to $\R$ satisfying  
\begin{enumerate}[(i)]
\item Monotonicity: $\EE[X]\le \EE[Y]$ if $X\le Y$;
\item Constant preserving: $\EE[c]=c$ for $c\in \R$;
\item Law invariance: $\EE[X]=\EE[Y]$ if $X$ and $Y$ are identically distributed.
\end{enumerate} 
In \cite{P19}, only the   first two properties are required in the definition. 
In optimal transport problems,  only the distribution of $c(X,Y)$ matters, because the constraints on $X$ and $Y$ are specified by marginal distributions; therefore, it is without loss of generality to 
assume law invariance.\footnote{Even if we use a function $\EE:\X\to \R$ that is not law invariant, the optimization in \eqref{eq:R1-1} 
is equivalent to using a law-invariant  function $\widehat \EE$ given by $\widehat \EE[Z]=\inf\{ \EE[Z']:Z'\mbox{ is identically distributed to }Z\}.$ Hence, requiring law invariance does not lose generality.  
}  
 A \emph{sub-linear expectation} $\EE$  is a non-linear expectation  that further satisfies 
\begin{enumerate}[(i)]
\item[(iv)] Subadditivity: $\EE[X+Y]\le \EE[X]+\EE[Y]$ for all $X,Y\in \X$; 
\item[(v)] Positive homogeneity: $\EE[\lambda X]=\lambda \EE[X] $ for $\lambda \ge 0$. 
\end{enumerate}
Similarly, a \emph{super-linear expectation} $\EE$  is a non-linear expectation  that   satisfies  positive homogeneity and $-\EE$ is subadditive.

We will focus on the popular class of non-linear expectations, the distorted expectations.
A  \emph{distortion function}  is an increasing (in the non-strict sense)    function $h: [0,1] \to [0,1]$ satisfying $h(0)=1-h(1)=0$. 
The set of distortion functions is denoted by $\mathcal H$. 
The \emph{distorted expectation} associated with distortion function $h$ is   $\EE^h: \X\to \R$ defined by
\begin{align}\label{yaari:theory}
	\EE^h[X]= \int_0^{\infty} h(\p(X> x  ))\d x + \int_{-\infty}^0 ( h(\p(X>x ))-1)  \d x.
	\end{align} 
	For an unbounded random variable $X$, a distorted expectation can also be defined via \eqref{yaari:theory}, which may be infinite or undefined. We focus on bounded random variables in this paper to illustrate the main ideas, while keeping in mind that most results also hold on more general spaces, provided that $\EE^h[X]$ is well defined.
	A distorted expectation is a non-linear expectation, and it is sub-linear if and only if $h$ is concave (see Lemma \ref{lem:1p} below). Moreover, any sub-linear expectation on $\X$ can be written as the supremum of sub-linear distorted expectations, as shown by \cite{K01}. 

		Distorted expectations are used in many different fields under different names, and they belong to the general class of Choquet integrals characterized by  \cite{S86, S89}. 
	We refer to the monographs of \cite{D94}, \cite{W10} and \cite{FS16}  for general  theories of the Choquet integral and distorted expectations in economics and finance. 
  
   	In decision theory, the distortion function $h$  
represents a subjective inflation/deflation of the true probability.
If the random variable $X$ is nonnegative, then \eqref{yaari:theory} reduces to
	\begin{align*}
	\EE^h[X]= \int_0^{\infty} h(\p(X>x))\d x.
	\end{align*}
If the distortion function $h$ is left-continuous, through a change of variable, $\EE^h$ can be written as (see e.g., \citet[Lemma 3]{WWW20})
\begin{equation}\label{distor2}
\EE^h[X]=\int_0^1 Q_{1-t}(X) \d h(t), 
\end{equation}
where $$Q_{p}(X) = \inf \{x\in \R: \p(X\le x)\ge p\},~p\in (0,1];$$ i.e., $Q_p(X)$ is the left-quantile of $X$ at $p\in (0,1]$.  
From \eqref{distor2}, it is clear that if $h$ is the identity, then $\EE^h[X]=\E[X]$, which is the expectation with respect to $\p$.   
The formula \eqref{distor2} will be frequently used  in the paper, and it has the natural interpretation that a distorted expectation puts different weights on different quantile levels. 
 The left-quantile $Q_p$ itself can be expressed as a distorted expectation, $Q_p=\EE^h$ where $h(t)= \id_{\{ t>1-p \}}$ for $t\in [0,1]$.
 Another common example of $\EE^h$, popular in finance, is 
the Expected Shortfall (ES)  at level $p\in [0,1)$,  defined as \begin{align}\label{eq:ES-def}
\ES_p(X)=\frac{1}{1-p} \int_p^1 Q_r(X) \d r\mbox{~~~for $X\in \X$.}
\end{align}
It holds that $\ES_p=\EE^h$ where $h(t)= \min\{t/(1-p),1\}$ for $t\in [0,1]$.  
Moreover, ES admits a dual representation (\citet[Theorem 4.52]{FS16}): 
\begin{align}\label{eq:ES}
\ES_p(X) =  \sup\left\{\E^P[X]:P\ll \p,~\frac{\d P}{\d \p}\le \frac{1}{1-p}\right\}\mbox{~~~for all $X\in \X$,}
 \end{align} 
 where $P\ll \p$ means that $P$ is absolutely continuous with respect to $\p$.
 

Next, we    describe three different   motivations for considering distorted expectations in optimal transport from decision theory, robust optimization, and risk management, respectively. 
\label{sec:motivation}
\begin{enumerate}
\item \textbf{Behavioral decision making.} Distorted expectations are most commonly used in decision theory, where decision makers maximize distorted expected utility of their random payoffs. The use of distorted expectations was developed by \cite{Q82}, \cite{Y87}, \cite{S89} and \cite{TK92} to address paradoxes that the classic expected utility theory (EUT) fails to accommodate. By now, theories based on distorted expectation, in particular, the rank-dependent utility theory (\cite{Q82,Y87}) and the cumulative prospect theory (\cite{TK92}),  are the standard alternatives to EUT in decision making under risk; see \cite{W10} for a general treatment.  The portfolio selection problem with  cumulative prospect theory is studied by \cite{HZ11}. 

To give a  concrete example, consider   a decision maker whose preference is
modelled by rank-dependent utility theory; that is, her welfare of a random payoff $Z$ is evaluated by $  \EE^h[u(Z)]$, where $\EE^h$ is a distorted expectation and $u$ is a utility function. 
Suppose that the decision maker can choose assignment between two known distributions $\mu$ and $\nu$,  for instance, pairing suppliers $X\lawis \mu$ and consumers  $Y\lawis \nu$ of a certain commodity. The welfare of the decision maker for the assignment is $\EE^h[u(f(X,Y))]$, where  $f(x,y)$ represents the profit of matching supplier $x$ with consumer $y$.
The use of the rank-dependent utility in this context means that the decision maker puts different weights on large and small profits (for instance, she may be more concerned about pairs with small profit, for a fairness consideration). 
 Therefore, the optimization problem of the decision maker is 
 $
\mbox{to maximize} ~  \EE^h[ u(f(X,Y))]$ subject to $ X\lawis \mu$ and $Y\lawis \nu,
 $
 which is a distorted transport problem.

%
 
\item \textbf{Robust optimization.} 
As a useful property of  distorted expectations, any convex (resp,~concave) distorted expectation can be written as the supremum (resp.~infimum)
  over a class of usual expectations under different probability measures (see Section \ref{sec:2} for details).
Taking the convex case as an example, the distorted transport problem is equivalent to solving for the optimal transport plan 
under the worst-case scenario within a class of probability measures.
This approach can be traced back  to Wald's minimax criterion (\cite{Wald1950}) and was axiomatized by \cite{GS89} in decision theory.  
This idea also aligns well with the idea of distributionally robust optimization,   studied by e.g.,
\cite{GS10}, \cite{WKS14} and \cite{BM19}. 

Consider a classic transport problem to minimize $\E[c(X,Y)]$ subject to $X\lawis \mu$ and $Y\lawis \nu$.
Suppose that the decision maker is  uncertain about the probability measure $\p$ she used to assess the transport cost (for instance, $\p$ may be estimated from data). Instead, she considers the robust objective
$$
\mbox{to minimize~}\sup_{P \in \mathcal Q} \E^P[c(X,Y)],
$$
where $\mathcal Q$ is a class of probability measures not far away from $\p$, called the uncertainty set. 
For many choices of $\mathcal Q$, the robust optimization problem can be converted into a distorted optimal transport problem.
As a concrete example, consider the uncertainty set $\mathcal Q=\{P\ll \p: \d P/\d \p \le 1/(1-\beta) \}$ where $\beta \in [0,1)$, that is, the set of all probability measures with a likelihood ratio with respect to $\p$  bounded by $1/(1-\beta)$, with $\beta=0$ corresponding to $\mathcal Q=\{\p\}$, the case of no uncertainty. By \eqref{eq:ES}, 
$$
\sup_{P\in \mathcal Q} \E^P[Z] = \ES_\beta (Z) \mbox{~for all random variables $Z$,}
$$ 
and hence, we arrive at a distorted transport problem. 
In general, any $\EE:Z\mapsto \sup_{P\in \mathcal Q} \E^P[Z]$ is a sub-linear expectation, provided that it satisfies law invariance.

\item \textbf{Robust risk aggregation.}  
In financial risk management, distorted expectations are known as distortion risk measures; see \cite{MFE15}. 
A robust risk aggregation problem typically concerns
$$
\sup\{\EE^h(X_1+\dots+X_n): X_1\lawis \mu_1,\dots,X_n\lawis \mu_n\},
$$
where $\mu_1,\dots,\mu_n$ are $n$ probability measures on $\R$. For robust risk aggregation problems in risk management,  see the survey of \cite{EPRWB14} and the monograph \cite{R13}. 
A robust risk aggregation problem for $\EE^h$ corresponds to a multi-marginal distorted transport problem,
and we focus on $n=2$ in this paper. 
Robust risk aggregation problems are motivated by the fact that the dependence structure of multiple risks is typically difficult to accurately specify. This  type of uncertainty is referred to {\it dependence uncertainty}, as studied by    \cite{EPR13, EWW15} and \cite{BLLW20}. 
Distorted expectations $\EE^h$ arise in this context because the standard  regulatory risk measures used in banking and insurance, such as Basel IV and Solvency II,
are the Value-at-Risk (VaR) and the ES, both of which are special cases of distorted expectations. Moreover, distortion risk measures are popular  in the literature of risk management. 
In the special case that $h$ is concave, robust risk aggregation can be   reformulated as a  multi-marginal classic optimal transport problem, as studied by \cite{EMNP22}.

\end{enumerate}

In the above settings,   distorted expectations $\EE^h$ appear 
for three different reasons: as preference models in the first setting,
as a result of statistical uncertainty in the second,
and as externally specified regulatory risk measures in the third. 
Moreover, although both the second and the third settings deal with uncertainty, their mathematical structures and interpretations are quite different. 
In robust optimization, uncertainty is reflected by the probabilities representing $\EE^h$, and there is no uncertainty associated with the transport plan;  in robust risk aggregation, uncertainty is reflected by the transport plan representing dependence among risks, and there is no uncertainty associated with the risk measure $\EE^h$.   

\section{The  framework and some preliminaries}  
\label{sec:2} 
We describe our main setting in this section.
Throughout, 
equalities and inequalities between random variables hold in the almost sure sense. 
For real numbers $u$ and $v$, we write $u \wedge v =\min\{u,v\}$,
$u \vee v =\max\{u,v\}$ and $u_+=\max\{u,0\}$.

\subsection{Non-linear and distorted optimal transport}

We first briefly explain the classic Monge-Kantorovich optimal transport problem, and then we focus on non-linear and distorted optimal transport on the real line. 

   Denote by $\Pi( \mathfrak X)$ the set of all Borel probability measures on a Polish space $\mathfrak X$ equipped with the Borel $\sigma$-field $\B( \mathfrak X)$. Consider Polish spaces $\mathfrak X , \mathfrak Y$, and probability measures $\mu\in \Pi(\mathfrak X) $ and $\nu\in\Pi ( \mathfrak Y)$.  
 We will always equip $\mathfrak X\times \mathfrak Y$ with the product $\sigma$-field. 
 Given a function $c:\mathfrak X\times \mathfrak Y\to \R$, which models the transportation cost, the classic optimal transport problem raised by Monge asks for the minimum of the expected transport cost
$ \int_{\mathfrak X} c(x,T(x))\mu(\d x)$ 
   for $T$ over the set $\T(\mu,\nu)$   of transport maps from $\mu$ to $\nu$,  i.e., measurable functions $T:\mathfrak X\to \mathfrak Y$ such that $\mu\circ T^{-1}=\nu$. 
  The   Kantorovich problem is a relaxation of Monge's problem, which solves for the minimum of 
    \begin{align}
     \int_{\mathfrak X\times  \mathfrak Y}c(x,y)\pi(\d x,\d y)
\label{eq:Kanto}     \end{align} 
    for $\pi$ over the set $\Pi(\mu,\nu)$ that is the set of transport plans from $\mu$ to $\nu$,  i.e.,  
      distributions on $\mathfrak X\times  \mathfrak Y$ with marginal $\mu$ on $\mathfrak X$ and marginal $\nu $ on $\mathfrak Y$.
Expressing the cost using random variables, the problem of \eqref{eq:Kanto} can be written as
    \begin{align} 
\mbox{to minimize~~}\E [c(X,Y)] \mbox{~~subject to $X\lawis \mu$ and $Y\lawis \nu$}.
\label{eq:prob}     \end{align}   
 Replacing $\E$ in \eqref{eq:prob} by a non-linear expectation $\EE$ or  $\EE^h$, we   formulate the problem of \emph{non-linear optimal transport} mentioned in the introduction:
   \begin{align} 
\mbox{to minimize~~}\EE [c(X,Y)] \mbox{~~subject to $X\lawis \mu$ and $Y\lawis \nu$}, 
\label{eq:R1-nonlinear}     \end{align} 
as well as the special case of  \emph{distorted optimal transport}  :
     \begin{align} 
\mbox{to minimize~~}\EE^h [c(X,Y)] \mbox{~~subject to $X\lawis \mu$ and $Y\lawis \nu$}.
\label{eq:distorted}     \end{align}  

 
 The problems  \eqref{eq:R1-nonlinear} and \eqref{eq:distorted}   are technically quite different from \eqref{eq:prob}.
In case $c$ is a linear combination of separable terms,  that is, $c(x+y)=f(x)+g(y)$ for some functions $f:\mathfrak X\to \R$ and $g:\mathfrak Y\to \R$, the problem \eqref{eq:prob} is trivial since $\E[c(X,Y)]=\E[f(X)]+\E[g(Y)] $, which does not depend on the dependence structure of $(X,Y)$ and thus  any transport plan is optimal. 
However,  even with a linear cost $c$, the problems  \eqref{eq:R1-nonlinear} and \eqref{eq:distorted}  are  highly nontrivial because the non-linear expectation is not an additive operator. 
In the case of additive cost,  we may, without loss of generality, replace  $f(X)$ and $g(Y)$  by  $X$ and $Y$, respectively, and consider the problem
    \begin{align} 
\mbox{to minimize~~}\EE [X+Y] \mbox{~~subject to $X\lawis \mu$ and $Y\lawis \nu$},
\label{eq:main}    \end{align}  
where $\mu$ and $\nu$ are probability measures on $\R$. 
 From now on, we will assume $\mathfrak X=\mathfrak Y=\R$.
 
In some contexts, such as robust risk aggregation (see Section \ref{sec:RA}),
 the maximization of $\EE[c(X,Y)]$ is relevant, which represents the worst-case risk under unknown dependence structure.  
The   maximization problem  of $\EE[c(X,Y)]$ 
can be converted to a minimization problem by taking the cost $-c$ and a different non-linear expectation (see  Lemma \ref{lem:2} below for the explicit case of $\EE^h$). Therefore, we can, without loss of generality, only study the minimization problems \eqref{eq:R1-nonlinear}-\eqref{eq:main}.

\subsection{Copulas}
To address problem \eqref{eq:prob} in the classic setting and our  problems \eqref{eq:R1-nonlinear}-\eqref{eq:main}, 
a convenient tool is the concept of copulas.
 For general references on copulas in risk management,  see \cite{N06} and \cite{MFE15}.  

  Let ${\mathcal I}=[0,1]$, and thus ${\mathcal I}^2$ is the unit square  ${\mathcal I}\times {\mathcal I}.$
A \emph{$2$-copula} (or briefly, a
copula) is a distribution function on $\mathcal I^2 $ with $\mathrm{U}[0,1]$ marginals.
Equivalently,  $C:{\mathcal I}^2 \to {\mathcal I}$ is a function with the following
properties:
For every $u,v \in {\mathcal I}$, 
 $C(u,0)=0=C(0,v)$, 
$C(u,1)=u$ and $C(1,v)=v$, and  
for every $u_1,u_2,v_1,v_2 \in \mathcal I$ with $u_1\le u_2$ and $v_1\le v_2$,
$C(u_2,v_2)-C(u_2,v_1)-C(u_1,v_2)+C(u_1,v_1)\ge 0.$ 
Throughout the paper, we denote  by $\mathcal C$ the set of all $2$-copulas.  

 Sklar’s theorem states that 
 for any joint distribution function $H$ 
with margins $F$ and $G$,  there exists a copula $C$ such that 
\begin{align}\label{eq:sk}
H(x,y) = C(F(x),G(y)) \mbox{ for all $x,y\in \R$.}
\end{align}
If $F$ and $G$ are continuous, then $C$ is unique. 
The copula $C$ in \eqref{eq:sk} is called a copula of $(X,Y)\lawis H$, 
and it is denoted by
  $C_{X,Y}$ if it is unique.

Transport plans, including transport maps, can be represented by copulas; for this reason, transport plans are also called   couplings in the literature. 
The optimization problem \eqref{eq:R1-nonlinear} can be equivalently formulated via copulas  
   \begin{align} 
\mbox{to minimize~~}\EE [c(F^{-1}_\mu (U), F^{-1}_\nu(V)) ] \mbox{~~subject to $ C\in \mathcal C$ where $(U,V)\lawis C$},
\label{eq:copula} 
   \end{align} 
   where  $F_\mu$ is the distribution function of $\mu$ with the corresponding left-quantile function $F^{-1}_\mu$, that is,
       $F_\mu(x)=\mu((-\infty,x])$ for $x\in \R$ and  
     $F^{-1}_\mu (t) = Q_t(X)$ for $t\in (0,1]$ with $X\lawis \mu$.

Two important copulas
that we will frequently encounter are    
the comonotonic copula (or the Fr\'echet-Hoeffding upper bound) $$C^+(u,v)=u\wedge v~\mbox{ for }~u,v\in [0,1],$$
and 
the  counter-monotonic copula (or the Fr\'echet-Hoeffding lower bound) 
$$C^-(u,v)=(u+v-1)_+~\mbox{ for }~u,v\in [0,1].$$

Two random variables $X$ and $Y$ are \emph{comonotonic}  if there exist increasing functions $f$ and $g$ such that $X=f(X+Y)$ and $Y=g(X+Y)$; they are \emph{counter-monotonic} if $X,-Y$ are comonotonic. 
It is well known that 
two random variables are comonotonic if and only if one of their copulas is $C^+$,
and two random variables are counter-monotonic if and only if one of their copulas is $C^-$.
Moreover, 
 $C^- \le C \le C^+$  for all $C\in \mathcal C$. This is known as the Fr\'echet-Hoeffding bounds; see \citet[Theorem 2.2.3]{N06}.
 The transport maps represented by $C^-$ and $C^+$ are presented in Figure \ref{fig:cc}.  
  
\begin{figure}[t]
\begin{center}
\begin{subfigure}[b]{0.4\textwidth}
\centering
 \begin{tikzpicture}
  \draw[-,thick] (3, 0) -- (7, 0);
  \draw[-,thick] (2.2, 2) -- (5.7, 2) ;
 \draw[->] (3.099077, 2) -- (3.725902, 0) ;
  \draw[->] ( 3.539025, 2) -- (4.348083, 0) ;
   \draw[->] (3.861236, 2) -- (4.803758, 0) ;
    \draw[->] (4.138764, 2) -- (5.196242, 0) ;
        \draw[->] (4.460975, 2) -- (5.651917, 0) ;
       \draw[->] (4.900923, 2) -- (6.274098, 0) ;
         
   \draw[domain=2.2:5.7, smooth, thick,variable=\x]  plot ( {\x},{exp(-(\x-4)*(\x-4)/(2*0.3))/sqrt(2*pi*0.3)+2.1}); 
  \draw[domain=3:7, smooth, thick,variable=\x]  plot ( {\x},{-exp(-(\x-5)*(\x-5)/(2*0.6))/sqrt(2*pi*0.6)-0.1}); 
   \filldraw (3.099077, 2) circle (1.5pt);
        \filldraw ( 3.539025, 2) circle (1.5pt);
           \filldraw(3.861236, 2)circle (1.5pt);
              \filldraw(4.138764, 2) circle (1.5pt);
              \filldraw (4.460975, 2) circle (1.5pt);
              \filldraw (4.900923, 2) circle (1.5pt);  
\end{tikzpicture}
\end{subfigure} 
\hspace{2em}
\begin{subfigure}[b]{0.4\textwidth}
\centering
\begin{tikzpicture}
  \draw[-,thick] (3, 0) -- (7, 0);
  \draw[-,thick] (2.2, 2) -- (5.7, 2) ;
 \draw[->] (3.099077, 2) -- (6.274098, 0) ;
  \draw[->] ( 3.539025, 2) -- (5.651917, 0) ;
   \draw[->] (3.861236, 2) -- (5.196242, 0) ;
    \draw[->] (4.138764, 2) -- (4.803758, 0) ;
        \draw[->] (4.460975, 2) -- (4.348083, 0) ;
       \draw[->] (4.900923, 2) -- (3.725902, 0) ;
         
   \draw[domain=2.2:5.7, smooth, thick,variable=\x]  plot ( {\x},{exp(-(\x-4)*(\x-4)/(2*0.3))/sqrt(2*pi*0.3)+2.1}); 
  \draw[domain=3:7, smooth, thick,variable=\x]  plot ( {\x},{-exp(-(\x-5)*(\x-5)/(2*0.6))/sqrt(2*pi*0.6)-0.1}); 
   \filldraw (3.099077, 2) circle (1.5pt);
        \filldraw ( 3.539025, 2) circle (1.5pt);
           \filldraw(3.861236, 2)circle (1.5pt);
              \filldraw(4.138764, 2) circle (1.5pt);
              \filldraw (4.460975, 2) circle (1.5pt);
              \filldraw (4.900923, 2) circle (1.5pt);  
\end{tikzpicture}
\end{subfigure} \end{center}
\caption{The transport maps between two normal distributions described by the copulas $C^+$ (left) and $C^-$ (right)}
  \label{fig:cc}
  \end{figure}
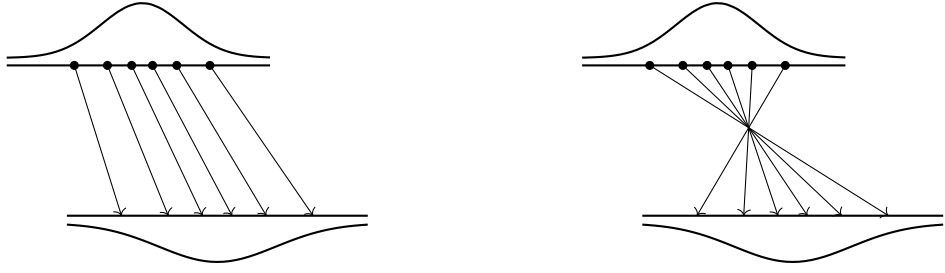

Another class of copulas that we will use is 
the class of ordinal sum copulas (\citet[Section 3.2.2]{N06}), indexed by $p\in (0,1)$, defined by 
  \begin{align}\label{cp}
\mathcal C_p = \big\{C\in \mathcal C: C(1-p,1-p)=1-p\big\}.
\end{align}   
One can verify  $\bigcap_{p\in (0,1)} \mathcal C_p=\{C^+\}$ (e.g., \citet[Theorem 4]{WZ21}) and $C^-\not\in \bigcup_{p\in (0,1)} \mathcal C_p$.
 
The copulas  $C^+$ and $C^-$ play special roles in classic transport problems \eqref{eq:prob}, equivalent to \eqref{eq:distorted} with $h$ being the identity.
A function $c:\R^d \to \R$ is \emph{submodular} if
 $$ 
c(\mathbf u \land \mathbf v) + c(\mathbf u \lor \mathbf v) \le  c(\mathbf u) +c(\mathbf v),\ \text{for
all } \mathbf u,\mathbf v \in \R^d,
 $$ 
where $\mathbf u \land \mathbf v$ is the component-wise
minimum of $\mathbf u$ and $\mathbf v$, and $\mathbf u \lor \mathbf v$ is the component-wise maximum of $\mathbf u$ and $\mathbf v$. A function $c$ is \emph{supermodular} if $-c$ is submodular.  
A simple example of a submodular cost function is $c(x,y)=|x-y|^p$ for $p\ge 1$.
For the classic transport problem, it is well known (see e.g.,  \cite{PW15}) that  if $c$ is submodular, then $C^+$ minimizes \eqref{eq:prob},
and if $c$ is a supermodular function, then $C^-$ minimizes \eqref{eq:prob}.
The class $\mathcal C_p$ solves a special case of \eqref{eq:distorted}, which we will see below.

With given $\EE$ and $c$, optimality of a copula $C$ for the problem \eqref{eq:copula} 
may depend on the choice of $\mu$ and $\nu$. 
However, in some interesting cases, the choice of $\mu$ and $\nu$ is irrelevant for optimality. 
For instance, as we see above, for the identity $h$ and submodular $c$, 
$C^+$ is optimal regardless of $\mu$ and $\nu$. 
This leads to our definition of universal optimality. 
In what follows, $X\laweq X'$ means that $X$ and $X'$ have the same law. 

\begin{definition}[Universal optimality]
	Let   $\EE$ be a non-linear expectation and $c:\R^2\to \R$. 
	The copula  $C $ is \emph{universally optimal} for $(c, \EE)$ if
	for all $X,Y\in \mathcal X$ with copula $C$, 
	\begin{align}\label{eq:1}
	\EE  [c(X,Y )] \le \EE  [c(X',Y')] \mbox{~~~for all $X',Y'\in \mathcal X$  satisfying $X\laweq X'$ and $Y\laweq Y'$}.	\end{align}
	If \eqref{eq:1} holds with $\EE=\EE^h$ for    $h\in \mathcal H$, we say that the copula $C $ is  {universally optimal} for $(c, h)$.
	\end{definition}
 We sometimes drop ``universally" in the text (but never in results) when the context is clear. 
	We also say that a random vector $(X,Y)$ is optimal if \eqref{eq:1} holds.
	In the special case that $c(x,y)=x+y$ which we focus on in later sections, we omit $c$ and say that 
	the copula  $C $ is  {universally optimal} for $\EE $ (or for $h$).  	
	As in the classic optimal transport literature, the term ``optimality" refers to  minimization. 
	We will replace ``optimal" with ``maximal" if  ``$\le$" in \eqref{eq:1} is replaced with ``$\ge$".
	Maximality can be useful when computing the worst-case dependence structure in risk aggregation problems, treated in Section \ref{sec:RA}.
	

	Universal optimality allows for the optimization problem \eqref{eq:copula} to be computed independent of the marginal distributions $\mu$ and $\nu$. This is arguably a quite strong condition, and universal optimizers may not exist. 
	However, universal optimizers do exist in some special and important cases. 
	For instance, as we see above, comonotonicity is universally optimal for {$\EE=\E$} and $c$ being submodular.  
For an example outside classic transport problems, 
take   $\EE^h (X) = -\ES_p(-X)$ for $X\in \X$ and $p\in (0,1)$; equivalently, $\EE^{\tilde h}=\ES_p$ (see Lemma \ref{lem:2}).
The recent result of \citet[Theorem 5]{WZ21} on risk aggregation of ES implies that
$C$ is universally optimal  for  $h $ if and only if $C\in \mathcal C_{1-p}$; in that result,
 the worst-case risk aggregation  problem is studied by maximizing $\ES_p$, which corresponds to minimizing  $\EE^h$.


\subsection{Properties of distorted expectations}
\label{sec:distortion}

We present some well-known properties of distorted expectations in the following three lemmas, 
which will be useful in our analysis. 
\begin{lemma}\label{lem:1}
For $h\in \mathcal H$, $\EE^h$ is additive for comonotonic random variables.
\end{lemma}
Lemma \ref{lem:1} follows from the fact that all Choquet integrals, including $\EE^h$, are additive for comonotonic random variables (e.g., \cite{S86}).

For random variables $X$ and $Y$, we say that 
  $X$ is {dominated by}  $Y$ in \emph{convex order}, denoted by $X \le_{\rm cx} Y$, if $\E[f(X)]\leq\E[f(Y)]$ for any convex function  $f$ provided that the expectations exist. 
A function $\rho:\X \to \R$ is \emph{increasing in convex order} if $\rho(X)\le \rho(Y)$ whenever $X\le_{\rm cx}Y$.
In the above definition, if ``convex function" is replaced with ``increasing convex function", ``concave function", or ``increasing concave function", then the three resulting orders are called increasing convex order, concave order, and increasing concave order, respectively.  
\begin{lemma}\label{lem:1p}
For $h\in \mathcal H$,  the following are equivalent: 
\begin{enumerate}[(i)]
\item 
  $\EE^h$ is subadditive (resp.~superadditive);
  \item 
  $\EE^h$ is increasing in convex order (resp.~concave order);
\item   $\EE^h$ is increasing in increasing convex order  (resp.~increasing concave order);
  \item  $h$ is concave  (resp.~convex).  
  \end{enumerate}
\end{lemma}
A more general version of 
Lemma \ref{lem:1p} can be found in Theorem 3 of \cite{WWW20}, where several other properties are considered and $h$ is not assumed monotone.
Point (iii) is not included in Theorem 3 of \cite{WWW20} but its equivalence to (ii) with increasing $h$ can be easily shown using standard arguments in stochastic order. 
As an example, for $\EE^h=\ES_p$ where $p\in (0,1)$, the function $h$ is concave, and hence $\ES_p$ is subadditive; this is a key property that makes  $\ES_p$ a coherent risk measure in the sense of \cite{ADEH99}.

The next lemma on sign changes of distorted expectations will also be useful.
For $h \in \mathcal H$,   let $\tilde h\in \mathcal H$ be given by
 $\tilde h(t)=1-h(1-t)$ for $t\in [0,1]$, which  is called the \emph{dual} of $h$. 
 If $h$ is continuous, then \eqref{distor2} can be rewritten as 
$$
\EE^h[X]=\int_0^1 Q_{1-t}(X) \d h(t) = \int_0^1 Q_t( X) \d \tilde h(t). 
$$
\begin{lemma}\label{lem:2}
For $h\in \mathcal H$ and $X\in \X$, $\EE^h (X)= -\EE^{\tilde h}(-X)$.
\end{lemma}
Lemma \ref{lem:2} can be found in \cite{DKLT12} and   \citet[Lemma 2]{WWW20}, which also holds for non-monotone $h$ in the latter reference. This lemma will be useful by allowing us to convert between a minimization problem and a corresponding maximization problem.

Finally,   a distorted expectation may be written as the worst-case or best-case value of expectation over a collection of probability measures. In particular,
  $h$ is concave if and only if  there exists a collection $\mathcal Q  $ of  probability measures on $(\Omega,\mathcal A)$ containing $\p$ such that 
 \begin{align} 
\EE^h[X] =  \sup_{P\in \mathcal Q}\E^P[X]\mbox{~~~for all $X\in \X$,}
\label{eq:worst-case} \end{align}
 and 
 $h$ is convex if and only if   there exists a collection $\mathcal Q  $  of probability measures on $(\Omega,\mathcal A)$ 
 containing $\p$ such that  \begin{align} 
\EE^h[X] =  \inf_{P\in \mathcal Q}\E^P[X] \mbox{~~~for all $X\in \X$.}
\label{eq:best-case} 
\end{align}
The representations \eqref{eq:worst-case} and \eqref{eq:best-case} 
are direct consequences of the fact that $\EE^h$ with a concave $h$ is a sub-linear expectation, and also a coherent risk measure (\cite{ADEH99}).
An example of such a representation is \eqref{eq:ES} in Section \ref{sec:motivation}. 
Using \eqref{eq:worst-case} and \eqref{eq:best-case}, a special case of  the distorted optimal transport   \eqref{eq:distorted} is to minimize the worst-case or best-case values of the expected cost over an uncertainty set, as   explained in Section \ref{sec:motivation}.

\section{General results}
\label{sec:3}

In this section, we present some general results on non-linear  and distorted optimal  transport.
In what follows, the cost function $c:\R^2\to \R$ is monotone if it is increasing or decreasing in the component-wise sense. 
\begin{theorem}\label{th:general}  
The copula $C^+ $ is universally optimal for any $(c,\EE)$ with super-linear $\EE$ and monotone submodular  $c$.
The copula $C^- $ is universally optimal for any  $(c,\EE)$ with  sub-linear $\EE$  and monotone supermodular    $c$.
\end{theorem}
\begin{proof}
We first consider the case that $\EE=\EE^h$.
Let $h$ be concave and $c$ be supermodular and monotone. 
Let $X,Y$ be comonotonic. 
We next show that, for $X'\laweq X$ and $Y'\laweq Y$, 
\begin{align}
\label{eq:icx} 
\E[ f \circ c(X,Y)] \ge \E[ f \circ c(X',Y')] \mbox{~~for all increasing convex $f$.}
\end{align} 
That is, $c(X,Y)$ dominates $c(X',Y')$ in increasing convex order. 
Note that $\E[g (X,Y)] \ge \E[g(X',Y')]$ for all supermodular functions $g$; see \citet[Theorem 3.9.8]{MS02}.
It remains to show that $f\circ c$ is supermodular. 
For any $\mathbf u,\mathbf v\in \R^2$, 
let $t=c(\mathbf u \land \mathbf v) + c(\mathbf u \lor \mathbf v) -   c(\mathbf u) -c(\mathbf v)  $.
We have  $t\ge 0$ because $c$ is supermodular.  
Monotonicity of $c$ implies  either $c(\mathbf u \land \mathbf v) \le  c(\mathbf u)  \le c(\mathbf u \lor \mathbf v)
$
or 
$c(\mathbf u \land \mathbf v) \ge  c(\mathbf u)  \ge c(\mathbf u \lor \mathbf v)
$.
Hence, there exists $\lambda \in [0,1]$ such that 
$ c(\mathbf u)
=\lambda c(\mathbf u \land \mathbf v)  + (1-\lambda) c(\mathbf u \lor \mathbf v)
$,
which implies 
$ c(\mathbf v)+t
=(1-\lambda)  c(\mathbf u \land \mathbf v)  +  \lambda c(\mathbf u \lor \mathbf v)
$. 
Convexity and increasing monotonicity of $f$ yield
 $$
f ( c(\mathbf u \land \mathbf v)) + f ( c(\mathbf u \lor \mathbf v)) \ge 
f ( c(\mathbf u ))+ f (c(  \mathbf v) +t )
 \ge
f ( c(\mathbf u) ) +f ( c(\mathbf v)),
$$
and thus $f\circ c$ is supermodular. 
Therefore, \eqref{eq:icx} holds.
By Lemma \ref{lem:1p}, $\EE^h$ is increasing in increasing convex order. 
Hence, the comonotonic vector $(X,Y)$ maximizes $ \EE^h[ c(X,Y)] $.
Using Lemma \ref{lem:2}, $(X,Y)$ minimizes $\EE^{\tilde h}[-c(X,Y)]$.
Note that $\tilde h$ is an arbitrary  convex distortion function
and $-c$ is an arbitrary monotone  submodular cost function. 
Therefore, we conclude that $C^+$ is universally optimal for any $(c,h)$ with $c$ monotone and submodular and $h$ convex.

Next, suppose that $\EE$ is a super-linear expectation. Using Theorem 4 of \cite{K01},\footnote{This result has a continuity assumption that is later removed by \cite{JST06} and \cite{D12}.} there exists a set $\mathcal H_0\subseteq \mathcal H$ of convex distortion functions such that $\EE[X] = \inf_{h\in \mathcal H_0} \EE^h[X]$ for $X\in \R$. Since $C^+$ is universally optimal for   $(c,h)$ with any convex $h$, 
it is also universally optimal for $(c,\EE)$.

The other statement is symmetric by  using the fact that $\E[g (X,Y)] \ge \E[g(X',Y')]$ for all submodular functions $g$, where $X,Y$ are counter-monotonic, $X'\laweq X$ and $Y'\laweq Y$.
\end{proof}

 Different from the classic transport problem,  the cost function $c$ is assumed to be monotone in Theorem \ref{th:general}. This excludes, in particular, the popular cost function $c(x,y)=|x-y|^p$.
Unfortunately, this restriction is essential and cannot be dropped. 
For instance, take $c(x,y)=x+y$ which is both supermodular and submodular. 
With given marginal distributions, 
a counter-monotonic random vector $(X,Y)$ minimizes $\ES_p(X+Y)$, which has a simple interpretation that a maximally hedged portfolio has the smallest risk. This optimality is implied by the second statement in Theorem \ref{th:general}.
On the other hand, if we take $c(x,y)=x-y$ which is also both supermodular and submodular,
then a \emph{comonotonic} random vector $(X,Y)$ minimizes $\ES_p(X-Y)$, with the same interpretation as above.
We can see that although $c$ is submodular in both cases, the optimal transport plans (copulas) are completely different, 
due to the fact that $c$ is not monotone in the second example.
This explains why monotonicity is important in distorted optimal transport problems. 

To explain why monotonicity is irrelevant to the classic transport problem,  for the linear expectation, optimality of  $(X,Y)$ for the cost $c$ is the same as that for the cost $c'(x,y)=c(x,y)+f(x)+g(y)$ with arbitrary functions $f$ and $g$, because monotonicity of the cost $c'$  can be amended by freely choosing $f$ and $g$.
However, this is not true for distorted transport problems, since $\EE^h[c'(X,Y)] \ne \EE^h[c(X,Y)]+\EE^h[f(X)]+\EE^h[g(Y)]$ in general, and monotonicity of the cost function cannot be adjusted freely by adding separable terms. 
 
Next, we state a weak form of duality. For any cost function $c$ and 
$\mu \in \Pi(\R)$,
let 
$$S_c=\{( \phi,\psi ) \mid \phi(x)+\psi(y) \le c(x,y) \mbox{~for all $x,y\in \R$}\},$$
and let $X^\mu$ be a random variable with distribution $\mu$.

\begin{theorem}
\label{th:dual}
Let $\EE$ be a super-linear expectation. For any cost function $c$, weak duality holds:
\begin{align}\label{eq:weak:dual}
\inf_{X\lawis \mu; Y\lawis \nu} \EE  [c(X,Y)]\ge \sup_{(\phi,\psi)\in S_c} \left\{ \EE [\phi(X^\mu)] + \EE [\psi(X^\nu)]\right\}.
\end{align} 
If $c$ is affine (i.e., $c(x,y)=\alpha+\beta x+\gamma y$ for some $\alpha,\beta,\gamma\in \R$) and $\EE=\EE^h$ for some $h\in \mathcal H$, then strong duality holds:
\begin{align}\label{eq:strong:dual}
\inf_{X\lawis \mu; Y\lawis \nu} \EE^h [c(X,Y)] = \sup_{(\phi,\psi)\in S_c} \left\{ \EE^h[\phi(X^\mu)] + \EE^h[\psi(X^\nu)]\right\}.
\end{align}
\end{theorem}
\begin{proof}
For any $(\phi,\psi)\in S_c$, by super-linearity, for any random variables $X\lawis \mu$ and $Y\lawis \nu$, we have 
$$ \EE [c(X,Y)] \ge  \EE [\phi(X)+\phi(Y)] \ge    \EE [\phi(X)]+\EE[\phi(Y)].
$$
Taking an infimum over $(X,Y)$  on the left-hand side and a supremum over $(\phi,\psi)\in S_c$ on the right-hand side yields \eqref{eq:weak:dual}.
If $\EE=\EE^h$ and $c$ is affine with $c(x,y)=\alpha+\beta x+\gamma y$, then we can take $\psi (x) =\alpha +\beta x$ and $\phi(y)= \gamma y$.
If $\beta\gamma\ge 0$, then, for comonotonic $X$ and $Y$, using comonotonic additivity of $\EE^h$ in Lemma \ref{lem:1}, 
$\EE^h [ \alpha+\beta X+\gamma Y ] = \EE^h[\alpha+\beta X] + \EE^h[\gamma Y].$ 
If $ \beta\gamma < 0$, then, for counter-comonotonic $X$ and $Y$, using the same property,
$\EE^h [ \alpha+\beta X+ \gamma Y ] = \EE^h[\alpha+\beta X] + \EE^h[ \gamma Y].$ 
In either case,  \eqref{eq:weak:dual} holds as an equality, and thus strong duality \eqref{eq:strong:dual} holds.
\end{proof}

The strong duality in Theorem \ref{th:dual} cannot be expected with more generality. 
On the right-hand side of \eqref{eq:weak:dual}, the non-linearity caused by $\EE$ is only on the marginal level,
whereas it is on the level of joint distribution for the primal formulation. 
Strong duality in the form of \eqref{eq:strong:dual} heavily relies on both comonotonic additivity of $\EE^h$ and affinity  of $c$.
The strong duality \eqref{eq:strong:dual} in Theorem \ref{th:dual} is not very useful, since for an affine cost  $c(x,y)=\alpha+\beta x+\gamma y$, the corresponding optimal transport can be explicitly obtained as either comonoconitiy or counter-monotonicity, depending on the sign of $\beta$ and $\gamma$.  
Even weak duality may fail if $h$ is not convex, as shown in the next example.
\begin{example}
Fix $p\in (0,1)$ and let $h (t)= \min\{t/(1-p),1\}$ for $t\in [0,1]$; that is, $\EE^h=\ES_p$. Note that $h$ is concave.
For $X\lawis \mu=\nu=\mathrm{U}[0,1]$, and $Y=1-X$, we have 
$$
\ES_p(X+Y) = 1< 1+p = \ES_p(X) +\ES_p(Y).
$$
Therefore, with $c(x,y)=x+y$, $\phi(x)=x$, $\psi(y)=y$, we can see that \eqref{eq:weak:dual} does not hold.
\end{example}

Generally, the computation of 
$$
\inf_{X\lawis \mu; Y\lawis \nu} \EE^h [c(X,Y)] 
$$
for non-convex and non-concave $h$ is a challenging task, and this holds true even for the linear cost $c(x,y)=x+y$. \cite{WXZ19} obtained the following simple lower bound
\begin{align}\label{eq:WXZ}
\inf_{X\lawis \mu; Y\lawis \nu} \EE^h [X+Y] \ge \EE^{  h^* }[X^\mu] + \EE^{ h^*}[X^\nu],
\end{align}
where $  h^*:=\sup \{g\le h \mid \mbox{$g$ is convex}\}$ is the convex envelope of $h$ (see Figure \ref{issnew} below), with equality holding if the marginal distributions $\mu$ and $\nu$ satisfy the so-called conditional joint mixability; see Theorem 4.3 of \cite{WXZ19}. 
The latter property is nontrivial, and strict inequality in \eqref{eq:WXZ} often holds; this may explain the difficulty to establish strong duality.

From next section on, we will focus on distorted expectations and a linear cost.

 \section{Linear cost with a strictly convex or concave distortion}
 \label{sec:4}

From now on, our main  objective  is to study universal optimality  with linear cost $c(x,y)=x+y$.  This setting also includes the more general case of separable cost, $c(x,y)=f(x)+g(y)$ for arbitrary measurable functions $f$ and $g$ because the latter can be converted into the simple linear case by marginal transformations, that is, replacing $\mu$ and $\nu$ with $\mu\circ f^{-1}$ and $\nu\circ g^{-1}$.

Recall that the copula  $C $ is  {universally optimal} for $h $ (omitting the cost function $c$) if
	for any $X,Y\in \mathcal X$ with copula $C$,  
	\begin{align*}
	\EE^h [X+Y] \le \EE^h [X'+Y'] \mbox{~~~for all $X',Y'\in \mathcal X$  satisfying $X\laweq X'$ and $Y\laweq Y'$}.
	\end{align*}  
The case that $h$ is neither convex nor concave is the most mathematically challenging, but in this section we first completely understand the case with convexity or concavity.

Below we present a necessary condition for a copula to be universally optimal, which will be useful in proving several other results in the paper. 
For fixed $h\in \mathcal H$ and $u,v\in [0,1]$, 
define the function $K^h_{u,v}:  [ C^-(u,v), C^+(u,v) ] \to [0,1]$ by 
	\begin{align}\label{eq:K}
	K^h_{u,v}(t) = h( 1 - t) + h(1-u-v+t).
	\end{align}
	
\begin{lemma}\label{lemma:1}
Suppose that $C\in \mathcal C$ is universally optimal for $h\in \mathcal H$.  We have $ K^h_{u,v}(C(u,v)  )  \le K^h_{u,v}(t)  $ for all   $ u,v \in [0,1]$ and  $ t \in [ C^-(u,v), C^+(u,v) ].$
\end{lemma}

\begin{proof} 
	Given $u,v \in [0,1]$, for any $ t \in[ C^-(u,v), C^+(u,v) ] $, there exists   $\lambda \in [0,1]$ such that $$ t = \lambda C^+(u,v) + (1-\lambda) C^- (u,v).$$  Define  by $ C_0 = \lambda C^+ + (1-\lambda) C^-  $, which satisfies   $C_0 (u,v) = t$, and let $s=C(u,v)$.
	
Take $\mu$ as  a Bernoulli distribution with $\mu(\{0\})=u$ 
and $\nu$ as  a Bernoulli distribution with $\nu(\{0\})=v$. 
Let    $X\sim  \mu  $ and $Y\sim \nu $ be such that  $C_0$ is a copula of $(X,Y)$.
The joint distribution of $(X,Y)$ is given by Sklar's theorem; see \eqref{eq:sk}.
Similarly, let $X^*\sim  \mu  $ and $Y^*\sim \nu $ be such that  $C$ is a copula of $(X^*,Y^*)$. 
Write 
 $Z=X+Y$ and $Z^*=X^*+Y^*$.
 It is straightforward to compute
	\begin{align*}
	\mathbb P (Z = 0) &=   C_0(u,v)=t,\\
	\mathbb P (Z = 1) &=  u + v - 2  C_0(u,v) =u + v - 2  t  ,  \\
	\mathbb P (Z = 2) &= 1 - u - v +  C_0(u,v) =1 - u - v +  t.
	\end{align*} 
	Consequently, we have $$\EE^h  [Z]  =\int_0^{\infty} h \circ \p(Z>z)\d z =
	h(1 - t) + h( 1 - u - v + t) = K^h_{u,v}(t) .$$
	Similarly, $ \EE^h  [Z^*]  = h(1 - s) + h( 1 - u - v + s) = K^h_{u,v}(s) $.   
	By the  optimality of $C$, we obtain  $\EE^h[Z^*]\le \EE^h[Z]$, and hence $ K^h_{u,v}(s) \le  K^h_{u,v}(t)$ for all $t \in [ C^-(u,v), C^+(u,v) ] $.  
\end{proof}

Lemma \ref{lemma:1} says that the function $K^h_{u,v}(t)$ attains its minimum on the interval $[ C^-(u,v), C^+(u,v) ]$ at the point  $C(u,v)$ for any $C$ that is optimal.

\begin{remark}\label{rem:lem3}
 The proof of Lemma \ref{lemma:1} only involves mixtures of Bernoulli distributions $\mu $ and $\nu$.
 Hence, the   conclusion on $K^h_{u,v}$ in Lemma \ref{lemma:1} holds true if    optimality of $C$ is  required only  for Bernoulli marginal distributions, instead of requiring it for all marginal distributions. 
\end{remark}

The next result  characterizes universal optimality in case $h$ is strictly concave or strictly convex.  Note that Theorem \ref{th:general} does not address the issue of uniqueness. 
 
\begin{theorem}\label{th:unique}

For a strictly convex $h\in \mathcal H$, $C^+$ is the unique universally optimal copula;
for a strictly concave $h\in \mathcal H$, $C^-$ is the unique universally optimal copula.
 
\end{theorem}
\begin{proof}
Both  optimality statements follow from Theorem \ref{th:general} by taking the linear cost $c(x,y)=x+y$. 
Below, we show the uniqueness statements.

First, let $h$ be strictly convex.  For $ r\in(0,1]$ and $t\in[(2r-1)_+,r]$, note that  $0\le1-2r+t\le 1-t\le 1$. Strict convexity of $h$ implies that 
the function $K^h_{r,r}$, that is, $K^h_{r,r}(t)= h( 1 - t) + h(1-2r+t)$,   is strictly decreasing on $[(2r-1)_+,r]$. Hence, $K^h_{r,r}(r)<K^h_{r,r}(t)$ for any $t\in[(2r-1)_+,r)$, and  $K^h_{r,r}(t)$ has the unique minimum point $t=r$ on the interval $[(2r-1)_+,r]$.
By Lemma \ref{lemma:1}, the optimality of $C$ yields 
$ C(r,r)=r$ for all $r\in(0,1],$
and this implies $C \ge C^+$. 
Since $C\le C^+$ also holds, we conclude that $C=C^+$.  

Next, let $h$ be strictly concave.  
Take $u,v\in [0,1]$ with $u+v\le1$, and write $r=(u+v)/2$. 
For any $t\in[0,r]$,  we have $0\le1-2r+t\le1-t\le1$.
Strict concavity of $h$ implies that  
the function $K^h_{u,v}$, that is, $K^h_{u,v}(t)= h( 1 - t) + h(1-u-v+t)$,   is strictly increasing on $[0,r]$.
 Hence, $K^h_{u,v}(0)<K^h_{u,v}(t)$ for any $t\in(0,r]$, and  $K^h_{r,r}(t)$ has the unique minimum point $t=0$ on the interval $[0,r]$.
By Lemma \ref{lemma:1}, the optimality of $C$ yields 
$ C(u,v)=0$ for all $u,v\in[0,1]$ with $u+v\le 1$.
This together with the $1$-Lipschitz continuity of $C$ (e.g., Theorem 2.2.4 of \cite{N06}) implies $C \le C^-$. 
Since $C\ge C^-$ also holds, we conclude that $C=C^-$.    
\end{proof}

 Theorem \ref{th:unique} yields that $C^+$ and $C^-$ are the unique universally optimal copulas for the two classes of distorted optimal transport problems with concave or convex $h$. 
 In the next section, we study the more intriguing case that $h$ is not convex nor concave.

\begin{remark}
In a different context,
a weaker version of the statement on $C^+$ in Theorem \ref{th:unique} 
can be found in \citet[Theorem 7]{C10}, where   $h$ is assumed continuously differentiable with $h'(0)<\infty$.
\end{remark}

  Theorem \ref{th:general} implies that  the copula $C^-$ is universally optimal for  a convex $h$ and 
the copula $C^+$  is universally optimal for  a concave $h$.  Below, we present a corresponding version of this result for maximizers instead of minimizers. This fact will be used later in the proof of another result.
\begin{proposition}\label{pr:trivial}
  The copula $C^+$ is universally maximal for  a concave $h$ and 
the copula $C^-$  is universally maximal for  a convex $h$.
\end{proposition} 
\begin{proof}
This proposition is a consequence of two facts. First, $X+Y$ has the largest  convex order when $X,Y$ are comonotonic and  $X+Y$ has the smallest   convex order when $X,Y$ are counter-monotonic; 
see Corollary 3.28 of \cite{R13}.
Second, $h$ is concave (resp.~convex) if and only if $\EE^h$ is increasing (resp.~decreasing) in convex order by Lemma \ref{lem:1p}.
These two facts  yield the desired statement. 
\end{proof}
A similar uniqueness statement to Theorem \ref{th:unique} can be made for the maximizers, and it is omitted.

%
%
%
%
%
%


 \section{Linear cost with an inverse S-shaped distortion}
 \label{sec:ISS}

 In this section we focus on inverse S-shaped  distortion functions for their plausibility in decision theory 
 and their intriguing mathematical properties regarding distorted optimal transport. 
We continue to consider the linear cost $c(x,y)=x+y$.

 \subsection{Inverse S-shaped distortion function and two classes of copulas}
First, we formally define this class of distortion functions.

\begin{definition}
	[{ISS distortion function}] \label{def:iss}A distortion function $h\in \cal H$ is \emph{inverse S-shaped} (ISS) if 
	it is continuous and  there exists a point $p\in (0,1)$ (called an \emph{inflection point})  such that $h$ is  concave on $(0,p)$ and   convex on $(p,1)$. 
	It is \emph{strictly ISS} if it is   strictly  concave on $(0,p)$ and strictly convex on $(p,1)$. 
\end{definition}
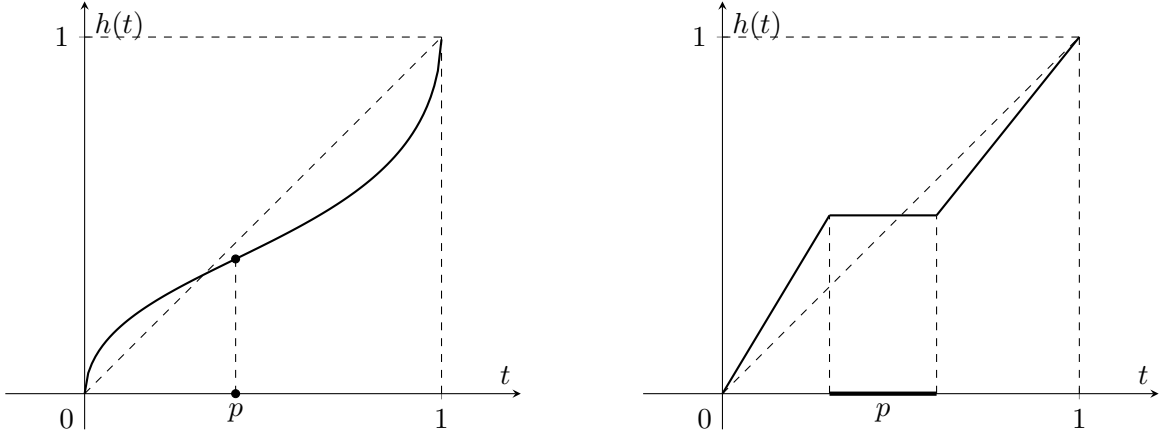
\begin{figure}[t]
\begin{center}
\begin{subfigure}[b]{0.45\textwidth}
\centering
\begin{tikzpicture}
\begin{axis}[
    axis lines = middle,
    xlabel = $t$,
    ylabel = {$h(t)$},
    domain=0:1,
    samples=100,
    ymin=-0.1,  
    ymax=1.1,
    xmin=0,  
    xmax=1,
    xtick={0,1},  
    ytick={0,1},  
    xticklabels={0,1},  
    yticklabels={0,1},  
    axis equal,  
]
\addplot [black, thick] {(x^0.6)/((x^0.6) + ((1-x)^0.6))^(1/0.6)};
\draw [dashed] (axis cs:0,0) -- (axis cs:1,1);
\draw [dashed] (axis cs:0,1) -- (axis cs:1,1);
\draw [dashed] (axis cs:1,0) -- (axis cs:1,1);
\node [below] at (axis cs:-0.05,-0.015) {0};  
\draw [dashed] (axis cs:0.4231,0.37779) -- (axis cs:0.4231,0);
\node [below] at (axis cs:0.4231,0) {$p$}; 
 \filldraw (axis cs:0.4231,0.37779) circle (1.5pt);
  \filldraw (axis cs:0.4231,0) circle (1.5pt);
\end{axis}
\end{tikzpicture}
\end{subfigure}
\hspace{2em}
\begin{subfigure}[b]{0.45\textwidth}
\centering
\begin{tikzpicture}
\begin{axis}[
     axis lines = middle,
    xlabel = $t$,
    ylabel = {$h(t)$},
    domain=0:1,
    samples=100,
    ymin=-0.1,  
    ymax=1.1,
    xmin=0,  
    xmax=1,
    xtick={0,1},  
    ytick={0,1},  
    xticklabels={0,1},  
    yticklabels={0,1},  
    axis equal,  %
]
\draw[black, thick] (axis cs:0,0) -- (axis cs:0.3,0.5);
\draw[black, thick] (axis cs:0.3,0.5) -- (axis cs:0.6,0.5);
\draw[black, ultra thick] (axis cs:0.3,0) -- (axis cs:0.6,0);
\draw[black, thick] (axis cs:0.6,0.5) -- (axis cs:1,1);
\draw [dashed] (axis cs:0,0) -- (axis cs:1,1);
\draw [dashed] (axis cs:0,1) -- (axis cs:1,1);
\draw [dashed] (axis cs:1,0) -- (axis cs:1,1);
\draw [dashed] (axis cs:0.3,0) -- (axis cs:0.3,0.5);
\draw [dashed] (axis cs:0.6,0) -- (axis cs:0.6,0.5);
\node [below] at (axis cs:0.45,0) {$p$}; 
\node [below] at (axis cs:-0.05,-0.015) {0};  

\end{axis}
\end{tikzpicture}
\end{subfigure}
\end{center}
\caption{ISS distortion functions that are strictly ISS (left, with a unique inflection point) and not strictly ISS (right, with many inflection points)}
\label{ISSf}
  \end{figure}
  
ISS distortion functions are also called \emph{cavex} (meaning ``concave-convex"); see \cite{W01}.
 Figure \ref{ISSf} presents two examples of ISS distortion functions.  
 The ISS distortion functions 
reflect  on the empirical observation that small probabilities of both very good and very bad events are both overweighed in decision making (\cite{TK92,GW99, Abdellaoui2000}).  
Clearly, a strictly ISS distortion function is strictly increasing and  has a unique inflection point. 


Next, we introduce two special classes of  copulas that are candidates for solving optimal transport in the context of ISS distortion functions. 
For any $p\in(0,1)$, define 
	\begin{align}\label{copucbstar}
&	C^\pm_p(u,v) = \left\{
\begin{array}{ll}
	 u\wedge v    & u\wedge v \in [0,1-p]  \\
	 (u+v-1) \vee (1-p) ~    &  { \mbox{otherwise};  }
\end{array}
\right. \\
\label{copucbstar2}
&	C^{\mp}_p(u,v) = \left\{
\begin{array}{ll}
	  ( u+v-1+p)_+\qquad &   u\vee v \in [0,1-p]   \\
	    u\wedge v        &  { \mbox{otherwise.}  }
\end{array}
\right.
\end{align}
Figure \ref{suppcp} presents the supports of $C^\pm_p$ and $C^\mp_p$, which are represented by some line segments in $[0,1]^2$. 
\begin{figure}[t]
\begin{center}
\begin{subfigure}[b]{0.45\textwidth}
\centering
 \begin{tikzpicture}
\draw [step = 0.5,help lines] ( 0,0 ) grid (5,5);
\fill (0,0) circle (2pt);
\fill (2,5) circle (2pt);
\fill (2,2) circle (2pt);
\fill (5,2) circle (2pt);
\fill (2,0) circle (2pt);
\fill (0,2) circle (2pt);
\draw [ thin ] (0,0) -- (5,0);
\draw [ thin ] (0,0) -- (0,5);
\draw [ thin ] (5,0) -- (5,5);
\draw [ thin ] (0,5) -- (5,5);
\draw [very thick] (0,0) -- node[ right] {} (2,2);
\draw [very thick] (2,5) -- node[ right] {} (5,2);
\node [below] at (2,0) {$1-p$}; 
\node [left] at (0,2) {$1-p$}; 
\end{tikzpicture}
\end{subfigure}
\begin{subfigure}[b]{0.45\textwidth}
\centering
 \begin{tikzpicture}
\draw [step = 0.5,help lines] ( 0,0 ) grid (5,5);
\fill (2,2) circle (2pt);
\fill (5,5) circle (2pt);
\fill (2,0) circle (2pt);
\fill (0,2) circle (2pt);
\draw [ thin ] (0,0) -- (5,0);
\draw [ thin ] (0,0) -- (0,5);
\draw [ thin ] (5,0) -- (5,5);
\draw [ thin ] (0,5) -- (5,5);
\draw [very thick] (2,0) -- node[ right] {} (0,2);
\draw [very thick] (2,2) -- node[ right] {} (5,5);
\node [below] at (2,0) {$1-p$}; 
\node [left] at (0,2) {$1-p$}; 
\end{tikzpicture}
\end{subfigure}
\end{center}
\caption{Supports of the copulas $C_{p}^\pm$  (left) and $C_p^\mp$ (right)}\label{suppcp}
\end{figure}
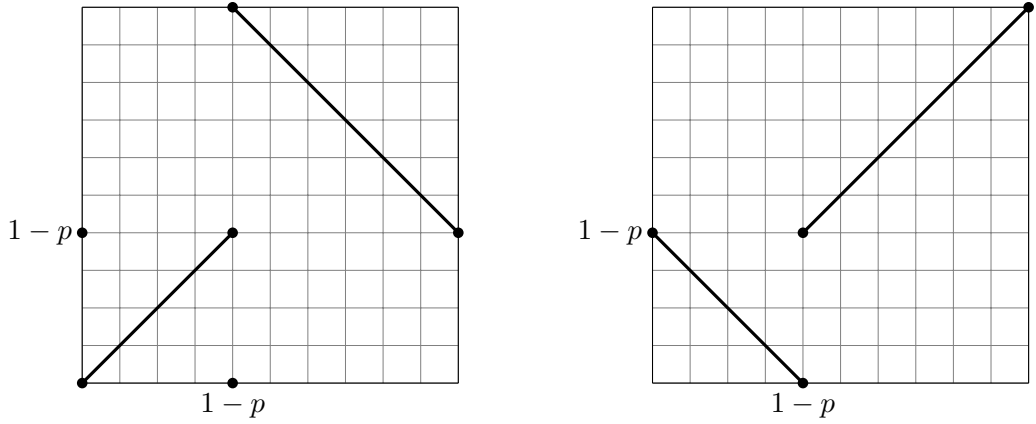

 Since $C^\pm_p (1-p,1-p)=C^\mp_p(1-p,1-p)=1-p$, 
  $C^\pm_p$ and $C^\mp_p $ are both in the set $\mathcal C_p$   defined in \eqref{cp}.
Indeed, $C^\pm_p$ and $C^\mp_p$ are ordinal sum copulas 
generated by $C^+$ and $C^-$ in different orders;  this explains our choice of  notation, with ``$\pm$" representing ``first $+$  then $-$" in the ordinal sum copula construction.   
The corresponding transport maps are illustrated in Figure \ref{fig:pmmp}.
A stochastic representation of  $C^\pm_p$, that is, a random vector
$(U,V)\lawis C^\pm_p$, is given by
\begin{align}
\label{eq:stochastic}
 U\lawis \mathrm{U}[0,1] \mbox{~and~} V=U\id_{\{U\le 1-p\}} + (2-p-U) \id_{\{U\ge 1-p\}}.
 \end{align}
A stochastic representation of  $C^\mp_p$ is given by
\begin{align}
\label{eq:stochastic2}
 U\lawis \mathrm{U}[0,1] \mbox{~and~} V=(1-p-U)\id_{\{U\le 1-p\}} +  U \id_{\{U\ge 1-p\}}.
 \end{align}

 
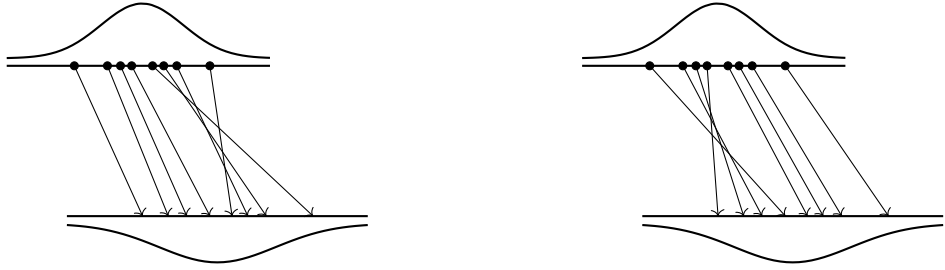
\begin{figure}[t]
\begin{center}
\begin{subfigure}[b]{0.4\textwidth}
\centering
 \begin{tikzpicture}
  \draw[-,thick] (3, 0) -- (7, 0);
  \draw[-,thick] (2.2, 2) -- (5.7, 2) ;
 \draw[->] (3.099077, 2) -- (4.007314, 0) ;
  \draw[->] ( 3.539025, 2) -- (4.348083, 0) ;
    \draw[->] ( 3.712774, 2) -- (4.593801, 0) ;
   \draw[->] (3.861236, 2) -- (4.902663, 0) ;
    \draw[->] (4.138764, 2) -- (6.274098, 0) ;
         \draw[->] (4.287226, 2) -- (5.651917, 0) ;
        \draw[->] (4.460975, 2) -- (5.406199, 0) ;
       \draw[->] (4.900923, 2) -- (5.196242, 0) ;
         
   \draw[domain=2.2:5.7, smooth, thick,variable=\x]  plot ( {\x},{exp(-(\x-4)*(\x-4)/(2*0.3))/sqrt(2*pi*0.3)+2.1}); 
  \draw[domain=3:7, smooth, thick,variable=\x]  plot ( {\x},{-exp(-(\x-5)*(\x-5)/(2*0.6))/sqrt(2*pi*0.6)-0.1}); 
   \filldraw (3.099077, 2) circle (1.5pt);
        \filldraw ( 3.539025, 2) circle (1.5pt);
          \filldraw( 3.712774, 2)  circle (1.5pt);
           \filldraw(3.861236, 2)circle (1.5pt);
              \filldraw(4.138764, 2) circle (1.5pt);
                \filldraw(4.287226, 2) circle (1.5pt);
              \filldraw (4.460975, 2) circle (1.5pt);
              \filldraw (4.900923, 2) circle (1.5pt);  
\end{tikzpicture}
\end{subfigure}
\hspace{2em}
\begin{subfigure}[b]{0.4\textwidth}
\centering

 \begin{tikzpicture}
  \draw[-,thick] (3, 0) -- (7, 0);
  \draw[-,thick] (2.2, 2) -- (5.7, 2) ;
 \draw[->] (3.099077, 2) -- (4.902663, 0) ;
  \draw[->] ( 3.539025, 2) -- (4.593801, 0) ;
    \draw[->] ( 3.712774, 2) -- (4.348083, 0) ;
   \draw[->] (3.861236, 2) -- (4.007314, 0) ;
    \draw[->] (4.138764, 2) -- (5.196242, 0) ;
         \draw[->] (4.287226, 2) -- (5.406199, 0) ;
        \draw[->] (4.460975, 2) -- (5.651917, 0) ;
       \draw[->] (4.900923, 2) -- (6.274098, 0) ;
         
   \draw[domain=2.2:5.7, smooth, thick,variable=\x]  plot ( {\x},{exp(-(\x-4)*(\x-4)/(2*0.3))/sqrt(2*pi*0.3)+2.1}); 
  \draw[domain=3:7, smooth, thick,variable=\x]  plot ( {\x},{-exp(-(\x-5)*(\x-5)/(2*0.6))/sqrt(2*pi*0.6)-0.1}); 
   \filldraw (3.099077, 2) circle (1.5pt);
        \filldraw ( 3.539025, 2) circle (1.5pt);
          \filldraw( 3.712774, 2)  circle (1.5pt);
           \filldraw(3.861236, 2)circle (1.5pt);
              \filldraw(4.138764, 2) circle (1.5pt);
                \filldraw(4.287226, 2) circle (1.5pt);
              \filldraw (4.460975, 2) circle (1.5pt);
              \filldraw (4.900923, 2) circle (1.5pt);  
\end{tikzpicture}
\end{subfigure}
\end{center} 
\caption{The transport maps between two normal distributions described by the copulas $C^{\pm}_p$ (left) and $C^{\mp}_p$ (right) with $p=0.5$}
\label{fig:pmmp}
\end{figure}

\subsection{Optimality within a subset}

 To address the optimal transport problem   for ISS distortion functions, 
we  need a modification of universal optimality constrained within a subset of copulas.  	Let $\mathcal C' \subseteq \mathcal C$ and  $h\in \mathcal H$.
	The copula  $C  \in \mathcal C'$ is \emph{universally optimal  within $\mathcal C'$} for $h$ if
	for any $X,Y\in \mathcal X$ with copula $C$, 
	we have
 $$ 
	\EE^h [X + Y] \le \EE^h [X'+Y'] 
 $$for all $X',Y'\in \mathcal X$ with a copula in $\mathcal C'$  satisfying $X\laweq X'$ and $Y\laweq Y'$.

The next proposition illustrates  the symmetric roles of $C^\pm_p$ and $C^\mp_p$ in \eqref{copucbstar}-\eqref{copucbstar2} within the set $\mathcal C_p$.
Recall that $\tilde h$ is the dual of $h$, and we speak of maximality, instead of optimality,  for a copula maximizing the distorted cost. 
\begin{proposition}\label{prop:symmetric}
Let $h\in \mathcal H$. 
The copula $C^\pm_p$ is universally optimal (resp.~within   $\mathcal C_p$) for $h$ if and only if
the copula $C^\mp_{1-p}$ is universally maximal (resp.~within  $\mathcal C_{1-p}$) for $\tilde h$. 
\end{proposition} 
\begin{proof}
We only show the ``only if" statement, and 
the ``if" statement  is symmetric. 
For a random vector
 $(X',Y')$, let  $(X,Y)$ have copula $C^\mp_{1-p}$ and the same marginal distributions as $(X',Y')$. Note that $(-X,-Y)$  has copula $C^\pm_{p}$. 
Optimality of $C^\pm_p$ implies 
$\EE^h[- X-Y] \le \EE^h [ -X'-Y']$. 
By using Lemma \ref{lem:2},
$$\EE^{\tilde h} [X+Y] =- \EE^h[-X-Y] \ge  -\EE^h [-X'-Y']
= \EE^{\tilde h}[X'+Y'],$$
and this gives   $\EE^{\tilde h} [X+Y] \ge \EE^{\tilde h}[X'+Y']$. 
Hence, $C^\mp_{1-p}$ is universally maximal for $\tilde h$.
Note that in the above argument, if $(X',Y')$ has a copula in  
$\mathcal C_{1-p}$, 
 then $(-X',-Y')$ has a copula in 
 $\mathcal C_{p}$, and hence the optimality statements within $\mathcal C_p$ and $\mathcal C_{1-p}$ hold true.
\end{proof}

In the next result, 
we show that $C^\pm_p$ is optimal within  $\mathcal C_p$ for an ISS distortion function  with an inflection point $p\in (0,1)$.  This, combined with  
Proposition \ref{prop:symmetric}, also shows that $ C^\mp_{p}$  is maximal within  $\mathcal C_p$ for such an ISS distortion function.
 The maximization problem is relevant  in the context of risk aggregation treated in Section \ref{sec:RA}. 
 
\begin{proposition}\label{th:1}
	Let  $h\in \cal H$ be an ISS distortion function with an inflection point $p\in (0,1)$. 
The copula $C^\pm_p$ is universally optimal  within $\mathcal C_p$ for $h$, and it is the unique optimizer if $h$ is strictly ISS. 
\end{proposition}

\begin{proof}   
Let 
 $(X,Y)$ be a random vector with a copula in $\mathcal C_p$.  Using the quantile representation of $\EE^h$ in \eqref{distor2},  we have 
		\begin{align}\label{eq:sum-p}
		\EE^h [X+Y] 
 =\int_0^{p}  Q_{1-t}(X+Y)     \d h (t)   + \int_{p}^1 Q_{1-t}(X+Y)      \d h(t).
		\end{align}
		The key observation here is that for copulas in the set $\mathcal C_p$, the quantile of $X+Y$  at level $t\in [0,p)$ is solely determined by the copula of $(X,Y)$ on $[0,p]^2$, and  the quantile of $X+Y$  at level $t\in (p,1]$ is solely determined by  the copula of $(X,Y)$ on $[p,1]^2$.  
		Therefore, to minimize \eqref{eq:sum-p}, it suffices to consider the two regions separately. 	The term $\int_0^{p}  Q_{1-t}(X+Y)   \d h (t)$ is minimized by counter-monotonicity since $h$ is concave on $[0,p]$, and  $\int_{p}^1  Q_{1-t}(X+Y)   \d h (t)$ is minimized by comonotonicity since $h$ is convex on $[p,1]$.
	Hence, $C_p^\pm$ is universally optimal within   $\mathcal C_p$ for $h$.
	For a strictly ISS  $h$,  the uniqueness statement of the optimal copula follows from the corresponding uniqueness statements for comonotonicity and counter-monotonicity with strictly convex and strictly concave distortion functions in Theorem \ref{th:unique}.
 \end{proof}
  
We have seen that $C_p^\pm$ is universally optimal within $\mathcal C_p$ for an ISS function $h$ with inflection point $p$. 
One may naturally wonder whether $C^\pm_p$ is  universally optimal also within $\mathcal C$ for such an $h$. 
It turns out that this is not always the case, as shown by the following example.

\begin{example}
\label{ex:non-optimal}

We consider  the following ISS distortion function 
$$ h(t) =  \left\{ 
\begin{array}{ll} 
2t - 2t^2  ~~~~ & t \in [0,0.5]  \\ 1 - 2t + 2t^2~~~~ & t \in [0.5,1],
\end{array} 
\right. $$
where the inflection point $p = 0.5$. 
For standard uniform random variables $X,Y,X',Y'$
where $(X,Y)$ has the copula $C^\pm_{0.5}$ and $(X',Y')$ has the copula $C^*$ depicted in Figure \ref{fig1} (right panel), 
we can compute $\EE^h[X+Y]\approx 0.9167 >  \EE^h[X'+Y']\approx 0.9087$;  we omit the detailed computation. 
Therefore, $C^\pm_{0.5}$ is not universally optimal for $h$. 

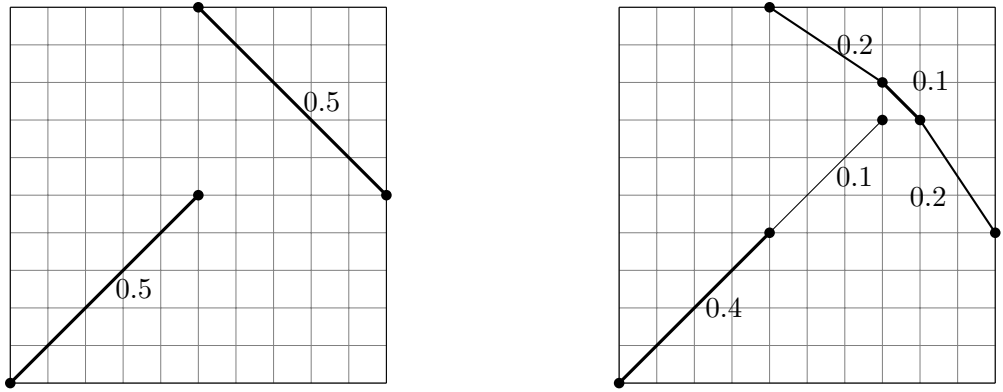
\begin{figure}
\begin{center}
\begin{subfigure}[b]{0.45\textwidth}
\centering
 \begin{tikzpicture}
\draw [step = 0.5,help lines] ( 0,0 ) grid (5,5);
\fill (0,0) circle (2pt);
\fill (2.5,2.5) circle (2pt);
\fill (2.5,5) circle (2pt);
\fill (5,2.5) circle (2pt);

\draw [ thin ] (0,0) -- (5,0);
\draw [ thin ] (0,0) -- (0,5);
\draw [ thin ] (5,0) -- (5,5);
\draw [ thin ] (0,5) -- (5,5);
\draw [very thick] (0,0) -- node[ right] {$0.5$} (2.5,2.5);
\draw [very thick] (2.5,5) -- node[ right] {$0.5$} (5,2.5);
\end{tikzpicture}
  \end{subfigure}
\hspace{1em}
\begin{subfigure}[b]{0.45\textwidth}
\centering
\begin{tikzpicture}
\draw [step = 0.5,help lines] ( 0,0 ) grid ( 5, 5);
\fill (3.5,4) circle (2pt);
\fill (4,3.5) circle (2pt);
\fill (2,5) circle (2pt);
\fill (5,2) circle (2pt);
\fill (2,2) circle (2pt);
\fill (3.5,3.5) circle (2pt);
\fill (0,0) circle (2pt);

\draw [ thin ] (0,0) -- (5,0);
\draw [ thin ] (0,0) -- (0,5);
\draw [ thin ] (5,0) -- (5,5);
\draw [ thin ] (0,5) -- (5,5);
\draw [very thick] (3.5,4) -- node[ above right] {$0.1$} (4,3.5) ;
\draw [thick] (2,5) -- node[ right] {$0.2$} (3.5,4);
\draw [thick] (4,3.5) -- node[below  left] {$0.2$} (5,2);
\draw [thin] (2,2) --  node[ right] {$0.1$} (3.5,3.5);
\draw [very thick] (0,0) --  node[ right] {$0.4$} (2,2);
\end{tikzpicture}
 \end{subfigure}
 \end{center}
 \caption{Supports of the copulas $C^\pm_{0.5}$ (left) and $C^*$ (right), where the numbers represent the corresponding weight of each line segment}\label{fig1}
\end{figure}

\end{example}

In what follows, we search for conditions on $h$ under which $C^\pm_p$ is universally optimal. 

\subsection{A characterization of universal optimality}

Next, we present the most technically advanced result in the paper, which establishes a necessary and sufficient condition on $h$ for the existence of a universally optimal copula, and identifies its form.
 In what follows, $h'(t^+)$ represents the right derivative of $h$ at $t\in [0,1)$, and similarly $h'(t^-)$ represents the left derivative at $t\in (0,1]$. For an ISS distortion function $h$, the left and right derivatives always exist.

The key condition on $h$ here is 
$h '(p^+)\ge h'(0^+)$, where $p$ is the inflection point of $h$. 
	This condition is rather restrictive. Since $h$ is ISS, this means that the slope of $h$ on $[0,p)$ is always smaller than or equal to the slope of $h$ on $(p,1]$; see Figure \ref{issnew}. 
	This condition imposes a kink at $p$, i.e., $h'(p^+)>h'(p^-)$, since the slope is decreasing on $[0,p)$ and increasing on $(p,1]$. Admittedly, such a condition is rarely observed for human behaviour as  extreme losses and gains are typically both weighted more than the middle part of the distribution (\cite{W10}).
	Nevertheless, we will see that this condition is necessary and sufficient for a universally optimal copula to exist if $h$ is strictly ISS.

\begin{figure}
\begin{center}
\begin{tikzpicture}
  \draw[->] (0, 0) -- (8.5, 0) node[below] {$t$};
  \draw[->] (0, 0) -- (0, 8.5) node[right] {$h(t)$};

    \draw[scale=8] (1,0) node [below]   {$1$};
     \draw[scale=8] (-0.02,0) node [below]   {$0$};
      \draw[scale=8] (0,1) node [left]   {$1$};
       \draw[scale=8] (0.85,0) node [below]   {$p$};
      
   \draw[scale=8, domain=0:0.85, smooth, thick,variable=\x]  plot ( {\x},{-70*\x*\x/81+126*\x/81});
    \draw[scale=8,  domain=0.85:1, smooth,thick,variable=\x]  plot ( {\x},{2200*(1-\x)*(1-\x)/729-1199*(1-\x)/486+1});

    \draw [blue,scale=8, domain=0:0.4023857, smooth,  variable=\x] plot ( {\x},{-70*\x*\x/81+126*\x/81});

                   \draw  [scale=8,blue, smooth] (0.4023857,   0.4860074)-- (1,1) ;
         
          \draw  [scale=8,red, smooth] (0,0)-- (0.85,  0.6978395) ;
        \draw [red,scale=8, domain=0.85:1, smooth,  variable=\x]  plot ( {\x},{2200*(1-\x)*(1-\x)/729-1199*(1-\x)/486+1});
        
 \draw  [scale=8,dash pattern={on 0.84pt off 2.51pt}]  (0,1) -- (1,1) ;
  \draw  [scale=8,dash pattern={on 0.84pt off 2.51pt}]  (1,0) -- (1,1) ;
    \draw  [scale=8,dash pattern={on 0.84pt off 2.51pt}]  (0.85,0) -- (0.85, 0.6978395)  ;

  \fill[scale=8] (0.85, 0.6978395) circle (0.25pt);
  \fill[scale=8] (0.4023857,   0.4860074) circle (0.25pt);
 
\end{tikzpicture}
\caption{A strictly ISS distortion function satisfying $h'(p^+)\ge h'(0^+)$ with its  convex envelope in red and  concave envelope in blue}
\label{issnew}
\end{center}
\end{figure}
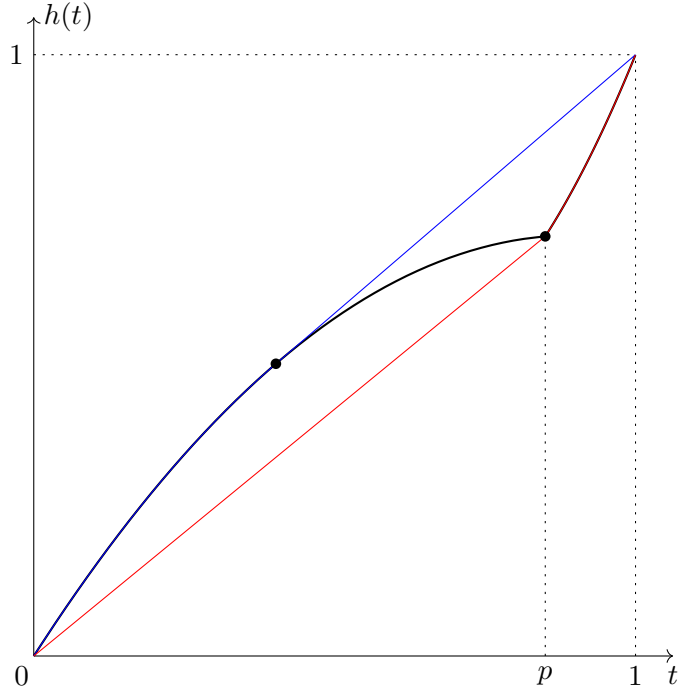

We first show  universal optimality of $C^\pm_p$ under the condition $h'(p^+)\ge h'(0^+)$,    a stronger optimality than 
the one obtained in Proposition \ref{th:1}. 
%
The proof of this result 
relies on some separate results on distorted optimal transport for conditional distributions. 
For this, let us introduce some notation. 
For any $Z\in \X$, let $F_Z$ be its distribution function, and take a uniformly distributed random variable $U_Z$ on $[0,1]$ 
such that $F_Z^{-1}(U_Z)=Z$ almost surely; the existence of $U_Z$ is guaranteed by, e.g., Lemma A.32 of \cite{FS16}.
Let $Z^\vee_p=F_Z^{-1}((1-p)+pU_Z)$; the random variable $Z^\vee_p$ is called the tail risk of $Z$ beyond its $(1-p)$-quantile by \cite{LW21}, and the distribution of $Z^\vee_p$ is called the $p$-tail distribution of $F$. 
Similarly, let  $Z^\wedge_p= F^{-1}_Z((1-p)U_Z)$, which represents the left tail risk.
Here, the superscripts $\vee$ and $\wedge$ reflect on the large and the small regions of the distribution. 
 With the notation above, we summarize some   results  needed for the proof of Theorem \ref{th:2} in the following two  lemmas.  
\begin{lemma}\label{lem:4-1}
Let  $h$ be an ISS distortion function with an inflection point $p\in (0,1)$ satisfying   $h
	'(p^+)\ge h'(0^+)$. 
  A decomposition of $\EE^h$ holds: for some $\theta,\theta_1,\theta_2\ge 0$ and convex $h_1,h_2\in \mathcal H$,
\begin{align}\label{eq:rwadd-2}
\EE^h [Z] = \theta \E[Z] - \theta_1 \EE^{h_1}[Z^\vee_p] +\theta_2 \EE^{h_2} [Z^\wedge_p], \mbox{~for all $Z\in \X$}.
\end{align}
  \end{lemma}

\begin{proof} 
	Take $\theta\in [ h'(0^+), h'(p^+)]$. 
	Let $g(t)=h(t)-\theta t$ for $t\in [0,1]$. Note that $g$ is concave on $[0,p]$ and convex on $[p,1]$.
	By definition, we have the relations $Q_t(Z^\vee_p)=   Q_{(1-p)+pt} (Z)$ and  $Q_t(Z^\wedge_p)=   Q_{(1-p)t} (Z)$ for $t\in (0,1)$. Using these relations, we have
		\begin{align*} 
	\EE^h[Z]&  = \theta \E[Z]+  \int_0^{p} Q_{1-t}(Z) \d g (t)   + \int_{p}^1  Q_{1-t}(Z)  \d g (t)     
\\& = \theta \E[Z]+  \int_0^{p} Q_{1-t/p}(Z^\vee_p) \d g (t)   + \int_{p}^1  Q_{(1-t)/(1-p)}(Z^\wedge_p)  \d g (t)   
\\& = \theta \E[Z]+  \int_0^{1} Q_{1-s}(Z^\vee_p) \d g_1 (s )   + \int_{0}^1  Q_{1-s}(Z^\wedge_p)  \d g_2 (s)  ,
	\end{align*}
	where $g_1(s)=g(ps)= h(ps)-\theta ps $ and $g_2(s)= g(p+(1-p)s) = h(p+(1-p)s)-\theta (p+(1-p)s)$ for $s\in (0,1)$.
	Let $g_3(s)= -g_1(s)$ and $g_4(s)=g_2(s)-h(p)+\theta p$. Since $h'(t^+) \le h'(0^+)\le \theta$ for $t\in (0,p)$ 
	and  $h'(t^+) \ge  h'(p^+) \ge  \theta$ for $t\in (p,1)$, we know that  
	$g_3$ and $g_4$ are increasing. Moreover, it is easy to see that $g_3$ and $g_4$ are convex and $g_3(0)=g_4(0)=0$.
	 Therefore, 
		\begin{align*} 
 \int_0^{1} Q_{1-s}(Z^\vee_p) \d g_1 (s )   + \int_{0}^1  Q_{1-s}(Z^\wedge_p)  \d g_2 (s)  
 = - \int_0^{1} Q_{1-s}(Z^\vee_p) \d g_3 (s )   + \int_{0}^1  Q_{1-s}(Z^\wedge_p)  \d g_4 (s).
	\end{align*}
	Letting $\theta_1=g_3(1)$ and $\theta_2=g_4(1)$ yields \eqref{eq:rwadd-2}.
\end{proof}
\begin{lemma}\label{lem:4-2}
Let $g$ be a convex distortion function and  $\mu$ and $\nu $ be probability measures on $\R$.
The random vector $(X,Y)$ with copula $C^\pm_p$ solves 
 both the optimization   problems  
 \begin{align} 
&\mbox{to maximize~~}\EE^{g} [(X+Y)^\vee_p] \mbox{~~subject to $X\lawis \mu$ and $Y\lawis \nu$}, 
\label{eq:main-pf}
\\
\mbox{and}~~~&\mbox{to minimize~~}\EE^{g} [(X+Y)^\wedge_p] \mbox{~~subject to $X\lawis \mu$ and $Y\lawis \nu$}.
\label{eq:main-pf2}
    \end{align}   
 
\end{lemma}

\begin{proof}
We first show the statement for the case of \eqref{eq:main-pf}. 
Note that the mapping $X\mapsto \EE^g[X_p^\vee]$ is a $p$-tail risk measure in the sense of \cite{LW21},
and we can use Theorem 3 of \cite{LW21}, which yields that the problem \eqref{eq:main-pf}
has the same maximum as
 \begin{align} \label{eq:main-pf3} \mbox{to maximize~~}\EE^{g} [X'+Y'] \mbox{~~subject to $X'\lawis \mu_p$ and $Y'\lawis \nu_p$}, 
\end{align}
 where $\mu_p$ is the $p$-tail distribution of $\mu$ and $\nu_p$ is the $p$-tail distribution of $\nu$.  
By Proposition \ref{pr:trivial}, \eqref{eq:main-pf3} is maximized  by counter-monotonicity of $X'$ and $Y'$.
By construction, $(X,Y)$ with copula $C^\pm_p$ and marginal distributions $\mu$ and $\nu$ satisfies 
$
\EE^g[(X+Y)^\vee_p] =   \EE^g[X'+Y']. 
$
Therefore, $(X,Y)$ with copula $C^\pm_p$ solves \eqref{eq:main-pf}. 

The proof for the case of \eqref{eq:main-pf2} is similar, by noting that comonotonicity minimizes $\EE^g $ of the sum of two random variables for a convex $g$ by Theorem \ref{th:general}.    
\end{proof}

The next result establishes that  $h
	'(p^+)\ge h'(0^+)$ is necessary and sufficient for the existence of a universally optimal copula when $h$ is strictly ISS. 

\begin{theorem}\label{th:2}
	Let  $h\in \cal H$ be an ISS distortion function with an inflection point $p\in (0,1)$ and $C$ be a copula.
Statement (i) implies statement (ii) below: 
	\begin{enumerate}[(i)]
	\item $C=C^\pm_p$ and $h
	'(p^+)\ge h'(0^+)$;
	\item   $C $  is universally optimal for $h$. 
	\end{enumerate}
  If $h$ is strictly ISS, then (i) and (ii) are equivalent. 
\end{theorem}
\begin{proof}
We first show the (i)$\Rightarrow$(ii) implication.
For $X,Y\in \X$, by Lemma \ref{lem:4-1},
 \begin{align}
 \label{eq:main-pf5} 
 \EE^h[X+Y] = \theta \E[X+Y] - \theta_1 \EE^{h_1}[(X+Y)^\vee_p] +\theta_2 \EE^{h_2} [(X+Y)^\wedge_p],\end{align}
 where $\theta,\theta_1,\theta_2 \ge 0$ and $h_1,h_2\in \mathcal H$ are convex. 
By Lemma \ref{lem:4-2} and the fact that $\E[X+Y]$ does not depend on the copula of $(X,Y)$,  the value of \eqref{eq:main-pf5}  is minimized by $\C^\pm_p$ under the constraints $X\sim \mu$ and $Y\sim \nu$. 
Therefore, (ii) holds. 
  
Next, for a strictly ISS $h$, we  prove the (ii)$\Rightarrow$(i) implication in three steps.
	\begin{enumerate}[(a)]
			\item First, we determine   $C(u,v)$ in case  $ u\wedge v \in [ 0,1-p]$.  
		If $0<u=v\le  (1-p)/2 $, let $r=u$. Then, $p+2r-1<0$, and for $t\in(0,r)$, we have $1-t>p$, $1-2r+t>p$ and $1-2r+t<1-t$. Strict convexity of $h$ on the interval $(p,1)$ implies that $K^h_{r,r}(t)= h(1-t)+ h( 1- 2r + t )$ is strictly decreasing on the interval  $ [0,r ]$ and thus strictly decreasing on $ [(2r-1)_+,r ]$. Hence,  $K^h_{r,r}$ has the unique  minimum point  $ t = r $  on the interval  $ [(2r-1)_+, r]$.
		By Lemma \ref{lemma:1},   we have
		\begin{align*}
		C(r,r) &=r \quad\mbox{for all}\quad r\in( 0,(1-p)/2).
		\end{align*} 
		If $ (1-p)/2\le u=v  \le 1-p$, again let $r=u$. Then, $p+2r-1\ge0$, and for $t\in(p+2r-1,r)$, we have $1-t>p$ and  $1-2r+t>p$. Strict convexity of $h$ on the interval $(p,1)$ implies that $K^h_{r,r}(t)= h(1-t)+ h( 1- 2r + t )$ is strictly decreasing on the interval  $ [ p+2r-1,r ]\subseteq  [(2r-1)_+,r ]$,
		and hence the minimum point of $K^h_{r,r}$  on the interval  $ [(2r-1)_+,r ]$  cannot be in  $(p+2r-1,r)$.  Lemma \ref{lemma:1} yields either $C(r,r)=r$ or $C(r,r)\le p+2r-1$. By continuity of $C$,  this implies
		\begin{align*}
		C(r,r) &=r \quad\mbox{for all}\quad r\in((1-p)/2,1-p).
		\end{align*}

	If $ 0 \le u \le 1-p$ and $u\le v$, then $u=C(u,u)\le C(u,v) \le C^+(u,v)= u$, which yields 
	 $
		C(u,v) = u= u\wedge v.
 $
Similarly, $C(u,v) = v=u\wedge v$ if $ 0 \le v \le 1-p$ and $v\le u$. 
To summarize,
		\begin{align*}
		C(u,v)& = u\wedge v, \qquad \mbox{ if }u\wedge v \le 1-p.
		\end{align*}
		
%
%
%
		
	\item Next, we determine $C(u,v)$ in case $u,v\ge 1-p$.	 
	If   $u+v\le 2-p$, then $0\le 2-p-(u+v)$.  Moreover,  $u\wedge v \ge 1-p$ implies $p\ge2-p-(u+v)$ and 
	$(u+v)/2\le u+v+p-1$. Strict concavity of $h$ on the interval $[2-p-(u+v),p ] $ implies that $K^h_{u,v}$ is strictly increasing on the interval $[1-p, (u+v)/2 ] $. Thus, the unique minimum point  of $K^h_{u,v}$  on the interval  $ [(u+v-1)_+, u\wedge v ]$ cannot be in $(1-p,(u+v)/2]$.  By Lemma \ref{lemma:1},  $C(u,v)\le 1-p$. Noting that $C(u,v)\ge C(1-p,1-p)=1-p$, we have
	\begin{align}\label{eq:rw-added1} C(u,v)=1-p, \qquad\mbox{ if } u,v\in (1-p, 1) \mbox{ and } u+v\le 2-p.
	\end{align} 
If $u+v\ge 2-p$, then $v\ge 2-p-u$. Since $C$ induces a measure, which is nonnegative on $[u,1]\times [2-p-u,v]$, we have $C(1,v)-C(1,2-p-u)-C(u,v)+C(u,2-p-u)\ge0$. 
Using \eqref{eq:rw-added1}, this implies 
 $v-(2-p-u)+1-p=u+v-1\ge C(u,v)$. On the other hand, since $C\le C^{-}$, we get $C(u,v)=u+v+1$. 
		Therefore, putting the above two cases together, we conclude \begin{align}\label{copulap}
  C(u,v)=(u+v-1)\vee(1-p) \qquad\mbox{ if } u\wedge v \ge 1-p.\end{align}
	\item The points (a) and (b) above shows that the optimal copula is $C=C^\pm_p$. 
It remains to show $h
	'(p^+)\ge h'(0^+)$. 
Let  $   u= 1-p $ and $ v = 1-\epsilon$, where $\epsilon$ is a small positive number satisfying $\epsilon <  p\wedge(1-p)$. By \eqref{copulap}, we have $ C(1-p,1-\epsilon ) = 1-p $. By Lemma \ref{lemma:1}, the function $t\mapsto K^h_{1-p,1-\epsilon}(t) = h(1-t) +  h( p +\epsilon-1 + t )  $ has a minimum point $ t =1-p $ on the interval $[1-p-\epsilon,1-p ]$. Hence, 
	$$K^h_{1-p,1-\epsilon}(1-p) = h(p) + h( \epsilon ) \le K^h_{1-p,1-\epsilon}(1-p-\epsilon) = h(p+ \epsilon ) + h(0), $$
	which implies $$\frac{ h(p + \epsilon )-h( p )  }{\epsilon} \ge  \frac{  h(\epsilon)-h(0)}{\epsilon} .$$ Letting $\epsilon \rightarrow 0+ $ yields $ h'(p^+ ) \ge h'(0^+ )  $. 

	\end{enumerate} 
	The above three steps  complete the proof of the implication (ii)$\Rightarrow$(i).
\end{proof}

The condition  $h'(p^+)\ge h(0^+)$ is quite strong in Theorem 4. This result has two main implications  that are useful for the  optimal transport theory. 
\begin{enumerate}[(a)]
\item If  $h'(p^+)\ge h'(0+)$ holds, then we can explicitly solve the distorted optimal transport problem  regardless of the marginal distributions, just as in the case of comonotonicity in  the classic transport theory, which optimizes submodular costs   regardless of the marginal distributions. 
\item If $h'(p^+)\ge h'(0+)$ does not hold, then there is no universal solution to the distorted optimal transport problem; that is, different marginal conditions lead to different solutions, which can be difficult to obtain in general.  
\end{enumerate}
Point (a) can be used to find solutions for some optimal transport problems, and an example is illustrated in Section \ref{sec:R1}. Point  (b)  helps to clarify to which extent   beyond the usual framework, we can still hope for universal optimizers. As a general message from Theorem \ref{th:2},
without convexity or concavity of $h$, such general optimizers can exist only in very restrictive cases.

%
%
%

An example of $h$ that does not satisfy the condition $h'(p^+)\ge h'(0^+)$ is presented in Example \ref{ex:non-optimal}, where we have seen that $C^\pm_p$ is not universally optimal for $h$.
In that example, $h'(p^+)=0<h'(0^+)=2$.

If the condition $h'(p^+)\ge h'(0^+)$ is not satisfied, then $C^\pm_p$ is not  universally optimal for $h$. 
Nevertheless, $C^\pm_p$ may still be an optimizer for some specific choices of $\mu$ and $\nu$.  
For instance, if both $\mu$ and $\nu$ are uniform on $[0,1]$, and the convex envelope of $h$ is linear on $(0,p)$, then 
the copula $C^\pm_p$ minimizes $\EE^h[X+Y]$ for $X\lawis \nu$ and $Y\lawis \nu$, as shown by \citet[Theorem 4.3]{WXZ19}.  

\begin{remark}\label{rem:max}
If we consider the maximization problem instead of minimization in Theorem \ref{th:2}, then by Proposition \ref{prop:symmetric}, the corresponding condition is $h'(p^-)\ge h'(1^-)$ and the corresponding maximizer is $C^{\mp}_p$.
Based on this observation, the  minimal or maximal cost can be directly computed from the stochastic representation in \eqref{eq:stochastic}-\eqref{eq:stochastic2}.
\end{remark}

%

\begin{remark} 
For a proof of $C=C^\pm_p$ in the $\Rightarrow$ direction of  Theorem \ref{th:2}, we  rely on Lemma \ref{lemma:1}, which only involves mixtures of Bernoulli distributions (see Remark \ref{rem:lem3}).  
Therefore, to pin down  $C^\pm_p$ as the only possible form for an optimal copula, it suffices to require optimality only for Bernoulli marginal distributions.
\end{remark}

\section{Risk measures, S-shaped distortion, and risk aggregation}
\label{sec:RA}

 As mentioned in Section \ref{sec:2},
the framework of distorted optimal transport includes
  robust risk aggregation problems with distortion risk measures, and we only consider the case of aggregating two risks in this paper. 

Although belonging to the same class of distorted expectations, the motivation and applications of  distortion risk measures in risk management are quite different
from the distorted expectations used for modeling preferences in decision theory. 
There are two main differences.
First, the most widely used distortion risk measures are  those that are convex (i.e., with a concave $h$, such as ES) and those  that have a simple distortion function, such as the VaR and the Range-Value-at-Risk (RVaR) defined below.
Second,  maximum for risk measures corresponds the worst-case, where as 
maximum for dual utilities corresponds to the best-case; therefore, the interpretations of maximization and minimization are quite different for the two settings.  In particular, in robust risk aggregation, one often focuses on the maximum of the risk measure, instead of the minimum.

The two popular classes of distortion risk measures, VaR and RVaR, are defined below.
VaR can be defined either as a left-quantile or a right-quantile in the literature.
We define VaR at level $p\in(0,1)$ as the corresponding right-quantile, that is,
 $$\VaR_p(X) =\inf\{x\in \R: \p(X\le x)>p\}, \mbox{ for $X\in \X$},$$
and this choice is because
the right-quantile admits a maximizing copula (and the left-quantile $Q_p$ admits a minimizing copula); see 
 \citet[Lemma 4.2]{BJW14}.
For    levels $p,q\in (0,1)$ with $p<q$, RVaR at level $(p,q)$ is defined as
$$
  \RVaR_{p,q} (X)= \frac{1}{q-p}\int_p^q Q_r(X) \d r, \mbox{ for $X\in \X$}.
$$
The distortion function for $\VaR_p$ is given by $h(t)= \id_{\{ t\ge 1-p \}}$ for $t\in [0,1]$
and the distortion function for $\RVaR_{p,q}$ is given by $h(t)= (t\wedge (1-p)-(1-q))_+/(q-p)$ for $t\in [0,1]$. 
 Figure \ref{fig:VaR-RVaR} depicts  the distortion functions of $\VaR_p$ and $\RVaR_{p,q}$. 
The VaR has been a standard risk measure in banking and insurance since the 1990s (see  \cite{J06}), and
RVaR was introduced by \cite{CDS10} as a robust risk measure.  
Both VaR and ES are limiting cases of RVaR, and  RVaR can be expressed as the inf-convolution of VaR and ES as shown by \cite{ELW18}.

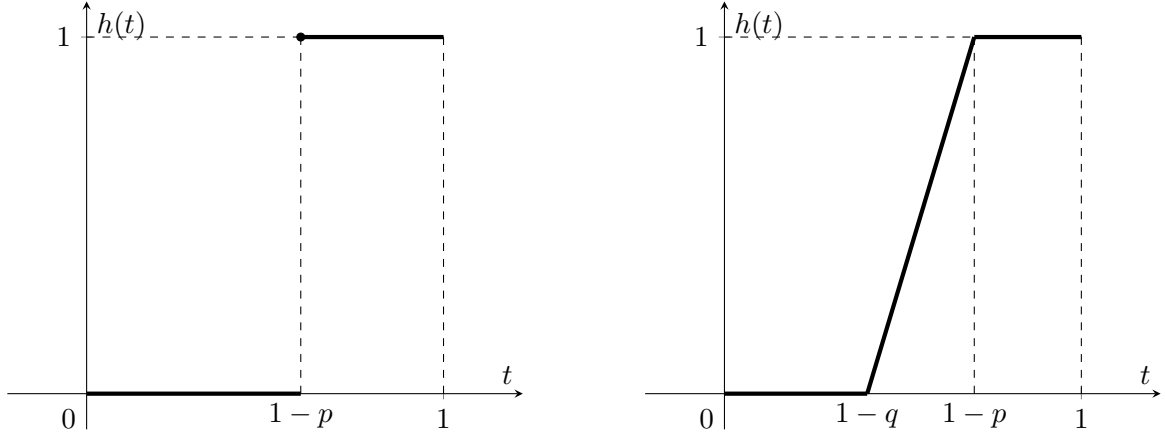
\begin{figure}[t]
\begin{center}
\begin{subfigure}[b]{0.45\textwidth}
\centering
\begin{tikzpicture}
\begin{axis}[
    axis lines = middle,
    xlabel = $t$,
    ylabel = {$h(t)$},
    domain=0:1,
    samples=100,
    ymin=-0.1,  
    ymax=1.1,
    xmin=0,  
    xmax=1,
    xtick={0,1},  
    ytick={0,1},  
    xticklabels={0,1},  
    yticklabels={0,1},  
    axis equal,  
]

 \draw  [ ultra thick] (axis cs:0,0) -- (axis cs:0.6,0) ;
  \draw  [ ultra thick] (axis cs:0.6,1) -- (axis cs:1,1) ;
\draw [dashed] (axis cs:0,1) -- (axis cs:1,1);
\draw [dashed] (axis cs:1,0) -- (axis cs:1,1);
\node [below] at (axis cs:-0.05,-0.015) {0};  
\draw [dashed] (axis cs:0.6,0) -- (axis cs:0.6,1);
\node [below] at (axis cs:0.6,0) {$1-p$}; 
 \filldraw (axis cs:0.6,1) circle (1.5pt);
 
\end{axis}
\end{tikzpicture}
\end{subfigure}
\hspace{2em}
\begin{subfigure}[b]{0.45\textwidth}
\centering
\begin{tikzpicture}
\begin{axis}[
     axis lines = middle,
    xlabel = $t$,
    ylabel = {$h(t)$},
    domain=0:1,
    samples=100,
    ymin=-0.1,  
    ymax=1.1,
    xmin=0,  
    xmax=1,
    xtick={0,1},  
    ytick={0,1},  
    xticklabels={0,1},  
    yticklabels={0,1},  
    axis equal,  %
]
 \draw  [ ultra thick] (axis cs:0.7,1) -- (axis cs:1,1) ;
  \draw  [ ultra thick] (axis cs:0,0) -- (axis cs:0.4,0) ;
  \draw  [ ultra thick] (axis cs:0.4,0) -- (axis cs:0.7,1) ;
\draw [dashed] (axis cs:0,1) -- (axis cs:1,1);
\draw [dashed] (axis cs:1,0) -- (axis cs:1,1);
\draw [dashed] (axis cs:0.7,0) -- (axis cs:0.7,1);
\node [below] at (axis cs:0.4,0) {$1-q$}; 
\node [below] at (axis cs:0.7,0) {$1-p$}; 
\node [below] at (axis cs:-0.05,-0.015) {0};  
\end{axis}
\end{tikzpicture}
\end{subfigure}
\end{center}
\caption{The distortion functions corresponding to $\VaR_p$ (left) and $\RVaR_{p,q}$ (right)}
\label{fig:VaR-RVaR}
  \end{figure}

An interesting observation from Figure \ref{fig:VaR-RVaR} is  RVaR has an S-shaped (instead of ISS) distortion function, and the distortion function of VaR is the limit of S-shaped ones.  
Since results for a concave $h$ are available from Theorems \ref{th:general} and \ref{th:unique} and Proposition \ref{pr:trivial},
below we will focus on the S-shaped distortion functions, and present a parallel result to Theorem \ref{th:2}.

\begin{definition}
	[{SS distortion function}] \label{def:ss}A distortion function $h\in \cal H$ is \emph{S-shaped} (SS) if 
	it is continuous and  there exists a point $p\in (0,1)$ (called an \emph{inflection point})  such that $h$ is convex on $(0,p)$ and  concave on $(p,1)$.
	It is \emph{strictly SS} if it is   strictly  convex on $(0,p)$ and strictly concave on $(p,1)$. 
\end{definition}

The following result characterizes universal optimality  for SS distortion functions. 
\begin{theorem}\label{pr:th2dual}
	Let  $h\in \cal H$ be an SS distortion function with an inflection point $p\in (0,1)$ and $C$ be a copula.
Statement (i) implies statement (ii) below: 
	\begin{enumerate}[(i)]
	\item $C=C^\pm_p$ and $h
	'(p^+)\le  h'(0^+)$;
	\item   $C $  is universally maximal for $h$. 
	\end{enumerate}
  If $h$ is strictly SS, then (i) and (ii) are equivalent. 
\end{theorem}

We omit the proof of Theorem \ref{pr:th2dual} as it is symmetric to Theorem \ref{th:2}. 
Nevertheless, in case $h$ has a bounded derivative, we can directly obtain 
Theorem \ref{pr:th2dual} from Theorem \ref{th:2}, and we explain this argument below.
Suppose that $h$ satisfies  $h
	'(p^+)\le  h'(0^+)$.
For $\lambda\in (0,1)$, define  the distortion function
$g_\lambda$   by $
g_\lambda(t) =  (t -   \lambda h(t)) /(1-\lambda). 
$
The assumption that $h$ has a bounded derivative implies that $g_\lambda$ is increasing for $\lambda>0$ small enough.
Since $h$ is SS,  we have that $g_\lambda$ is an ISS distortion function,
with $g_\lambda
	'(p^+)\ge  g_\lambda'(0^+)$. 
By Theorem \ref{th:2},  the  distorted transport problem for $g_\lambda$ is  minimized by $C^\pm_p$. Since $\EE^{g_\lambda} =( \E - \lambda \EE^h)/(1-\lambda)$ and the term $\E$ does not depend on the transport plan,
we know that   the transport problem for $\EE^h$ is   maximized by $C^\pm_p$. 
 The uniqueness statement  in Theorem \ref{pr:th2dual} also follows from the same observation. 

Corresponding statements on minimizers can be obtained similarly as discussed in Remark \ref{rem:max}. 

For $p,q\in (0,1)$ with $p< q$, any point in $[1-q,1-p]$ is an inflection point of the distortion function $h$ of $\RVaR_{p,q}$.
Taking the inflection point $r=1-p$ of $h$,  
the condition  $h
	'(r^+)=0= h'(0^+)$ holds.
	Therefore, the condition in Theorem \ref{pr:th2dual} holds,
	and a universally maximal copula is $C^{\pm}_r$;
	thus, we obtain a maximizing  dependence structure in the  $\RVaR_{p,q}$ risk aggregation problem; a related result on this problem is Theorem 3 of \cite{LW21}. This copula is not unique since the distortion function of RVaR is not strictly SS; indeed, by definition, $\RVaR_{p,q}$ neglects  the region of the distribution beyond its $q$-quantile and below its $p$-quantile, leaving some flexibility of the optimal copula. 
	
	As the limiting case of $\RVaR_{p,q}$ as $p\uparrow q$, the $\VaR_q$ risk aggregation problem is also maximized by $\C^{\pm}_{1-q}$. This fact was shown by \cite{M81} and \cite{R82}.  
	As the other limiting case of $\RVaR_{p,q}$ as $q\to 1$, the $\ES_p$ risk aggregation problem is   maximized by $ C^+$, which is also implied by Theorem \ref{th:general}. All optimizers above are not unique; for instance, $\mathcal C_{1-p}$ is the set of all maximizers for $\ES_p$ (\cite{WZ21}).

\section{An application}
\label{sec:R1}

 In this section, we illustrate the results in Theorems \ref{th:2} and \ref{pr:th2dual} in an example of economic production. 
Suppose that  a manufacturer produces a certain type of products in a given period of time. Each individual product requires a worker (e.g., an engineer) and a co-worker (e.g., a designer) to collaborate. 
The company has an equal number of workers and co-workers, and needs to match each worker $x\in \mathfrak X$ with a co-worker $y\in \mathfrak Y$,
where $\mathfrak X$ and $\mathfrak Y$ are two arbitrary index sets.  
For simplicity, suppose that the quality of a product has the additive form $f(x)+g(y)$, where $f:\mathfrak X\to\R$ is the production function of worker $x$ and  $g:\mathfrak Y\to\R$ is the production function of  co-worker $y$. 
We assume measurability of $f$ and $g$, and that the workers $x$ and the co-workers $ y$ are modelled by two   distributions $\mu$ and $\nu$.
The company assesses that the 20\% products with the worst quality can only be sold at the minimum price,
and  the 30\% products with the best quality are more than good enough to be sold at a maximum price. 
For this reason,
the company's main interest is  
the average quality of their products between 20\% and 70\% quantiles.\footnote{We assume the number of workers and coworkers is very  large so that it does not hurt to use continuous distributions and precise probability levels.} 

The company's optimization problem is to maximize the range-average quality, with $p=0.2$ and $q=0.7$, 
\begin{equation}\label{eq:R1-app1}
\RVaR_{p,q}(f(X)+g(Y)) \mbox{~~~subject to $X\lawis  \mu$ and $Y\lawis  \nu$,}
\end{equation}
where $\RVaR_{p,q}$ is defined in Section \ref{sec:RA}.
The problem in \eqref{eq:R1-app1} can be equivalently formulated as 
\begin{equation}\label{eq:R1-app2}
\RVaR_{p,q}(Z+W) \mbox{~~~subject to $Z\lawis \tilde \mu$ and $W\lawis  \tilde \nu$,}
\end{equation}
where $\tilde \mu=\mu\circ f^{-1}$ and $\tilde \nu=\nu\circ g^{-1}$. 

To solve this problem, as discussed in Section \ref{sec:RA}, note that $\RVaR_{p,q}=\EE^h$, where $h$ is an SS-distortion function 
with inflection point $r=1-p=0.8$. 
Indeed, any point in $[1-q,1-p]$ is an inflection point (see Figure \ref{fig:VaR-RVaR}), 
but only $r=1-p$ satisfies the condition in Theorem  \ref{pr:th2dual}.
By using Theorem  \ref{pr:th2dual}, a maximizer to \eqref{eq:R1-app2} 
is   $C^{\pm}_{0.8}$. The interpretation of this transport plan (see Figure \ref{suppcp}) 
is that,  the worst $20\%$ workers 
and the worst $20\%$ co-workers are matched together following a comonotonic pattern  to produce the low-quality products, and the best $80\%$ workers 
and the best $80\%$ co-workers
are matched together following a counter-monotonic pattern, so that the medium-range products between 20\% and 70\% quantiles have good quality. 
Our universal optimality guarantees this solution regardless of the skill distributions $\tilde \mu$ and $\tilde \nu$.

Some numerical examples are  reported in Table \ref{tab:1} below, where we compute 
the value of $\RVaR_{p,q}(Z+W)$ in \eqref{eq:R1-app2} for five different couplings, and some continuous marginal distributions. Because we have explicit methods to simulate from each copula, we   calculate each number in Table \ref{tab:1} by   the average of the 20\% and 70\% sample quantiles from a Monte Carlo simulation of sample size 10,000,000.  From the results, we can see that the optimal copula leads to a  substantially larger range-average quality.
\begin{table}[ht]
    \centering\def\arraystretch{1.5}
    \begin{tabular}{ c|c|c } 
 \multirow{2}{*}{Coupling} &   \multicolumn{2}{c}{$ \tilde \mu$ and $\tilde \nu$}    \\
   &  $ \mathrm{Beta}(5,2)$, $ \mathrm{uniform}[0,1]$ &  $ \mathrm{Expo}(1)$, $\text{Log-normal}(0,1)$ \\   
        \hline        
        Comonotonicity   & 1.1585    & 1.5717  \\
                \hline
        Counter-monotonicity &   1.2017   &  2.0210  \\
        \hline
        Independence  &   1.1644 & 1.9123 \\
        \hline
            Sub-optimal copula $C^{\pm}_{0.5}$  &  1.2123 &   2.0304 \\
        \hline
       Optimal copula $C^{\pm}_{0.8}$   &1.2931  &  2.3824 \\
        \hline
    \end{tabular}
    \caption{The  range-average quality    $\RVaR_{0.2,0.7}(Z+W)$ where $Z\lawis \tilde \mu$ and $W\lawis \tilde \nu$}
    \label{tab:1}
\end{table}

\section{Concluding remarks}

\label{sec:conclusion}

The framework of non-linear and distorted optimal transport is proposed, and we obtain many results in this framework. 
As our main findings, we established universal optimality and duality for some classes of non-linear  expectations and cost functions,
and fully characterized universal optimality for ISS and SS distortion functions with linear cost.
The optimal copula is also illustrated with an application of economic production.
Since the framework is new, there are many open questions, and we discuss some of them below.

First, we have focused on transport on the real line, and the structure of the real line is used extensively.
For the applications that motivated the paper, behavioral decision making, robust optimization and robust risk aggregation, the real line is the most natural space, but for other applications, more general spaces are needed.  
To extend our results to $\R^d$ with $d\ge 2$ or  general Polish spaces requires thorough separate study.

Second, we established weak duality in Theorem \ref{th:dual} only in case $\EE$ is super-linear. 
The case of a general non-linear expectation is difficult to analyze, and we only have a simple bound \eqref{eq:WXZ} for $\EE=\EE^h$.
Similarly, strong duality is only established for convex $h$ and affine cost, and we wonder whether it also holds in more general situations. 

Third, universally optimal couplings are obtained under the assumption that the distortion functions  are convex, concave, ISS, or SS. 
Although these distortion functions are the most commonly used in the application domains, we wonder whether 
universally optimal couplings can be obtained, or proved to be non-existent, for more general forms of $h$. 
Moreover, optimal couplings for specific marginal distributions, other than all marginal distributions, are left for future research.

Fourth, most results  in our paper are obtained for distorted expectations or sub-linear/super-linear expectations. General non-linear expectations   that are neither sub-linear nor super-linear will require substantially different techniques.

Fifth, the distorted optimal transport can be generalized to other settings other than the standard one we considered. 
Motivated by different applications, classic transport problems have been generalized
to, among many others, the multi-marginal setting (see \cite{RR98,P15}), transport between vector-measures (see \cite{WZ22}),  transport between  capacities (see \cite{GSZ23}), and the constrained settings, such as those of martingale transport (\cite{BHP13,GHT14, BJ16}), supermartingale transport (\cite{NS18}) or directional transport (\cite{NW22}). Each of these generalizations can be combined with our framework of distorted transport. 
We expect them to lead to great additional   challenges and new mathematics on optimal transport in the future.

\subsection*{Acknowledgements}
RW acknowledge financial support from Canada Research Chairs (CRC-2022-00141)
and the
Natural Sciences and Engineering Research Council of Canada (RGPIN-2018-03823, RGPAS-2018-522590).
 

\end{document}